%% file: Ancient-RF-new.tex
\documentclass[11pt, a4paper,reqno]{amsart}

\usepackage{amsmath}
\usepackage{amsfonts}
\usepackage{amssymb}
\usepackage{amsthm}
\usepackage{enumerate}
\usepackage[dvips]{graphics}

\newtheorem{thm}{Theorem}[section]
\newtheorem{cor}[thm]{Corollary}
\newtheorem{lem}[thm]{Lemma}
\newtheorem{prop}[thm]{Proposition}
\newtheorem{claim}[thm]{Claim}
\newtheorem{defn}[thm]{Definition}
\theoremstyle{remark}
\newtheorem{rem}[thm]{Remark}
\numberwithin{equation}{section}

\theoremstyle{definition}

\newcommand{\R}{\mathbb R}

\newcommand{\bu}{\bar u}

\newcommand{\tv}{v_\infty}
\newcommand{\e}{\epsilon}

\newcommand{\dds}{ \Delta_{S^2}}
\newcommand{\ds}{\nabla_{S^2}}

\newcommand{\al}{\alpha}

\newcommand{\la}{\lambda}

\newcommand{\ins}{\int_{S^2}}

\newcommand{\diam}{\mathrm{diam}}

\newcommand{\dist}{\mathrm{dist}}
\newcommand{\injrad}{\mathrm{injrad}}

\newcommand{\csch}{\mathrm{csch}}
\newcommand{\ar}{\mathrm{area}}

\newcommand{\bbu}{\bar u}
\newcommand{\bbbu}{\bar {\bar u}}
\newcommand{\buk}{\bar u_k}
\newcommand{\bv}{\bar v}

\def\XXint#1#2#3{{\setbox0=\hbox{$#1{#2#3}{\intr}$}
     \vcenter{\hbox{$#2#3$}}\kern-.5\wd0}}

\begin{document}

\title[Ancient Compact Solutions to the Ricci Flow on Surfaces]
{Classification of  Ancient Compact solutions \\ to the Ricci flow on Surfaces}

\author[P. Daskalopoulos]
{Panagiota Daskalopoulos$^*$}

\address{Department of
Mathematics, Columbia University, New York,
 USA}
\email{pdaskalo@math.columbia.edu}

\author[R. Hamilton]
{Richard Hamilton}

\address{Department of
Mathematics, Columbia University, New York,
 USA}
\email{hamilton@math.columbia.edu}

\author[N. Sesum]
{Natasa Sesum$^{**}$}
\address{Department of Mathematics, Columbia University, New York,
USA}
\email{natasas@math.columbia.edu}

\thanks{$*:$ Partially supported
by NSF grant 0604657}
\thanks{$**:$ Partially supported by NSF grant 0905749}
\maketitle 

\begin{abstract}
We consider an ancient solution $g(\cdot,t)$ of the Ricci flow  on a compact surface
that exists  for $t\in (-\infty,T)$ and becomes spherical at time  $t=T$. 
We prove that   the metric  $g(\cdot,t)$ is either a family of
contracting spheres,    which is a type I ancient solution,  or a    King-Rosenau solution,  which is  a type II ancient
solution.   
\end{abstract}

\section {Introduction}

We consider an ancient solution of the Ricci flow
\begin{equation}
\label{eqn-ricci}
\frac{\partial g_{ij}}{\partial t} = -2 \, R_{ij}
\end{equation}
on a compact two-dimensional surface   that exists for time   $t\in (-\infty,T)$
and becomes singular at  $t=T$,  for some $T < \infty$. In two dimensions we have   $R_{ij} = \frac 12 R\, g_{ij}$, where $R$ is
the scalar curvature of the surface. Moreover, on an ancient 
non-flat solution we have $R >0$.  
It is  well known (\cite{Ch}, \cite{Ha1}) that the surface also becomes extinct  at  $T$ and
it becomes spherical, which means that after a normalization, the
normalized flow converges to a spherical metric, to which we will refer as to the limiting sphere.

Since $R > 0$, by the Uniformization theorem and the fact that the Ricci flow in dimension two preserves the conformal class, we can parametrize the Ricci flow  by the limiting sphere at time $T$, that is, we can write
$$g(\cdot,t) = u(\cdot,t)\, g_{S^2}.$$
The spherical metric can be written as 
\begin{equation}
\label{eq-spher-met}
g_{S^2} = d\psi^2 + \cos^2\psi \, d\theta^2
\end{equation}  
where   $\psi,\theta$ denote the  global coordinates on the sphere.
An easy computation shows that (\ref{eqn-ricci}) is equivalent to 
the following evolution equation for the conformal factor $u(\cdot,t)$, namely 
\begin{equation}\label{eqn-U}
u_t = \Delta_{S^2} \log u -2\qquad \mbox{on} \,\, S^2 \times (-\infty,T)
\end{equation}
where $\Delta_{S^2}$ denotes the Laplacian on $S^2$.
Let us  recall, for future references,   that the only nonzero Christoffel
symbols for the   spherical metric (\ref{eq-spher-met}) are
$$\Gamma_{12}^2 = \Gamma_{21}^2 = -\tan\psi, \quad  \Gamma_{22}^1 = \frac{\sin 2\psi}{2}$$
where we use the indices  $1, 2$ for the $\psi, \theta$ variables respectively. 
It follows that for any function $f$ on the  sphere we have
$$\Delta_{S^2} f = f_{\psi\psi} - \tan\psi\, f_{\psi} + \sec^2\psi \,  f_{\theta\theta}$$
which,  in the case of a radially symmetric function $f=f(\psi)$,  becomes 
$$\Delta_{S^2} f = f_{\psi\psi} - \tan \psi \, f_\psi.$$

\medskip
We will assume, throughout this paper,  that $g= u\, ds_p^2 $ is an ancient solution to the
Ricci flow \eqref{eqn-U}  on the  sphere which becomes extinct  at time $T=0$.

It is natural to consider the pressure function $v = u^{-1}$ which
evolves by
\begin{equation}
v_t = v^2 \, (\dds \log v +2)\qquad \mbox{on} \,\, S^2 \times (-\infty,0)
\end{equation}
or, after expanding the laplacian of $\log v$, 
\begin{equation}\label{eqn-p1}
v_t = v \, \dds  v - |\ds v|^2 + 2 v^2\qquad \mbox{on} \,\, S^2 \times (-\infty,0). 
\end{equation}

\medskip

\begin{defn}
We will say that an ancient solution to the Ricci flow (\ref{eqn-ricci}) on
a compact surface  $M$ is of type I,  if it satisfies
$$\limsup_{t\to-\infty} \, ( |t| \, \max_M R(\cdot,t) )< \infty.$$ A solution which is not of type I,  will be called of type
II.
\end{defn}

Explicit examples of ancient solutions to the Ricci flow in two dimensions are:
\begin{enumerate}[i.]
\item
{\bf The contracting spheres}\\
They are described on $S^2$ by a pressure $v_S$ that is given by 
\begin{equation}\label{eqn-spheres}
v_S(\psi,t) = \frac{1}{2(-t)}
\end{equation}
and they are  examples of ancient type I shrinking Ricci solitons.
\item
{\bf The King-Rosenau solutions}\\ They were  discovered 
 by J.R. King (\cite{K1}, \cite{K2}) and later,  independently, by  P. Rosenau (\cite{R}).
They are described on $S^2$ by a pressure 
$v_K$ that has the form 
\begin{equation}\label{eqn-King-Rosenau}
v_K(\psi,t) = a(t) - b(t)\, \sin^2\psi
\end{equation}
with $a(t) = -\mu \, \coth( 2 \mu  t)$, $b(t) = -\mu \, \tanh(2 \mu  t)$, for some $\mu  >
0$. These  solutions are {\em  not
solitons}. We can visualize  them  as two cigars ''glued'' together
to form a compact solution to the Ricci flow. They are type II ancient
solutions.
\end{enumerate}

\medskip
Our goal in this paper is to prove the following classification result: 

\begin{thm}
\label{thm-class-anc}
Let $g=u\, g_{S^2}$ be an ancient  compact solution to
the Ricci flow \eqref{eqn-ricci}.  Then   $u$ is either one of  the contracting spheres   or 
one of the King-Rosenau solutions.
\end{thm}

\begin{rem}
The classification of two-dimensional,  complete, non-compact  ancient solutions of the Ricci flow 
 was recently  given in \cite{DS} (see also in \cite{Ha3}, \cite{SC}). The result in 
Theorem \ref{thm-class-anc} together with the results in \cite{DS} and \cite{SC} provide
a complete classification of all ancient two-dimensional complete solutions to  the Ricci flow, with the scalar curvature uniformly bounded at each time-slice.

\end{rem}

The  outline of the paper is as follows: 

\begin{enumerate}[i.]

\item In section \ref{sec-apriori} we will show a' priori  derivative estimates  on 
any  ancient solution $v$ of  \eqref{eqn-p1}, which hold  uniformly in time, up  to $t=-\infty$. These estimates will play a crucial role throughout 
the rest of the paper. 

\item In section \ref{sec-Lyapunov} we will introduce a suitable Lyapunov functional  and
we will use it to  show that the  solution $v(\cdot,t)$ of  \eqref{eqn-p1} converges,  as $t \to - \infty$, in
the  $C^{1,\alpha}$ norm, to a steady state $v_\infty$.

\item Section \ref{sec-limit} will be devoted to the classification of all
backward limits $v_\infty$.  We will show that  there is a
parametrization of the flow by a sphere,  in which $v_\infty(\psi,\theta) =
\mu\, \cos^2\psi$,  for some $\mu \ge 0$ ($\psi,\theta$ are the global coordinates on $S^2$).
When $\mu>0$,  then $v_\infty$ represents the cylindrical metric. 

\item In section \ref{section-KR} we will show that if $v_\infty(\psi,\theta) =
\mu\, \cos^2\psi$,    with $\mu >0$, then $v$ must be one of the King-Rosenau solutions.

\item In the last section \ref{section-sphere} we will show that if $v_\infty \equiv 0$,  then the 
solution $v$ must be  one of the  contracting spheres.

\end{enumerate}

\medskip

\noindent{\bf Acknowledgments.} We are indebted   to the referees of this article for their  valuable  suggestions and
comments. 

\subsection{Change of variables}

Throughout the paper we will be performing different changes of variables that we summarize below.
We may  write the evolving metric
$$g(\cdot,t) = u(\cdot,t) g_{S^2} = \bar{u}(\cdot,t) g_e = {\hat u}(\cdot,t) g_{c}$$
where $g_{S^2}, g_e, g_{c}$ denote  the metrics  on the standard sphere $S^2$, the euclidean plane and the cylinder, respectively. Denote by $\psi, \theta$ the global coordinates on $S^2$, by $r, \theta$ the polar coordinates on the plane and by $s, \theta$ the global coordinates on the cylinder. Mercator's projection gives us a relation between the sphere and the cylinder, that is, 
$$g(\cdot,t) = u(\psi,\theta,t) \, (d\psi^2 + \cos^2\psi \, d\theta^2) = {\hat u}(s,\theta,t)\,  (ds^2 + d\theta^2)$$
with 
\begin{equation}
\label{eqn-uU}
{\hat u}(s,\theta,t) = u(\psi,\theta,t)\cos^2\psi\qquad \cos  \psi = (\cosh s)^{-1}.
\end{equation}
It follows that
\begin{equation}\label{eqn-cv}
\cosh s = \sec \psi \qquad \mbox{and} \qquad \sinh s = \tan \psi,
\end{equation}

\noindent A simple computation shows  that if $u$ is a solution of \eqref{eqn-U}, then 
${\hat u}$ satisfies the  equation
\begin{equation}
\label{eqn-hatu}
{\hat u}_t = \Delta_c \log {\hat u}\qquad \mbox{on} \,\,\, \R \times [0,2\pi] \times (-\infty,0)
\end{equation}
where  $\Delta_c$ is the  cylindrical Laplacian, that is 
$$\Delta_c f = f_{ss} + f_{\theta\theta}$$
for a function $f$ defined on $\R \times [0,2\pi]$. 
Furthermore, if 
$$g(\cdot,t) = \bar{u}(r,\theta,t) (dr^2 + r^2\, d\theta^2)$$ where $(r,\theta)$ denote the polar coordinates of the plane
obtained by projecting $S^2 \setminus \{ \psi = \frac \psi 2 \}$ or $S^2 \setminus \{ \psi = -\frac \psi 2 \}$ via stereographic projection, it is
easy to compute that 
\begin{equation}
\label{eqn-Ubaru}
{\hat u}(s,\theta,t) = r^2 \, \bar{u}(r,\theta,t), \qquad r = e^s.
\end{equation}
The function  $\bar{u}$ satisfies the  equation
\begin{equation}
\label{eqn-baru}
\bar{u}_t = \Delta \log\bar{u}\qquad \mbox{on} \,\,\, \mathbb{R}^2\times (-\infty,0)
\end{equation}
where $\Delta$ is the euclidean Laplacian in polar coordinates. 

\noindent Equivalently, the  pressure functions  in cylindrical and polar coordinates given by  
$$\hat v(s,\theta,t) := \frac{1}{{\hat u}(s,\theta,t)}\quad \mbox{and} \quad  \bar{v}(r,\theta,t) := \frac{1}{\bar{u}(r,\theta,t)}$$ 
satisfy the  evolution equations
$$\hat v_t = \hat v\, \Delta_c \hat v - |\nabla \hat v|^2\qquad \mbox{on} \,\,\, \mathbb{R}^2\times (-\infty,0)$$
and
$$\bar{v}_t = \bar{v}\, \Delta\bar{v} - |\nabla\bar{v}|^2\qquad \mbox{on} \,\,\, \mathbb{R}^2\times (-\infty,0)$$
respectively. 

The change of variables between the sphere $S^2 \setminus \{ \psi = \frac \pi 2\}$ and the plane 
which maps the south pole $\psi = -\pi/2$ to the origin, is given via stereographic projection, that is,
$$r = \frac{1+\sin\psi}{\cos\psi}$$
and if  $g(\cdot,t) = u(\psi,\theta,t) g_{S^2} = \bar{u}(r,\theta,t) g_e$, then 
$$\bar{u}(r,\theta,t) = \frac{\cos^4\psi}{(1+\sin\psi)^2} \, u(\psi,\theta,t).$$

\section{A' priori estimates}
\label{sec-apriori}

We will assume,  throughout this section,  that $v$ is an ancient solution of
the Ricci flow \eqref{eqn-p1} on $S^2 \times (-\infty,0)$ which becomes extinct  at $T=0$.
We fix $t_0 < 0$. We   will establish a'priori derivative estimates  
on $v$ which hold uniformly  on $S^2 \times (-\infty,t_0]$. 
 We will denote by $C$
various constants which depend on $t_0$ and may vary from line to line but they are always 
independent of time $t \in (-\infty,t_0]$. 

Since our solution is  ancient, the scalar curvature 
$R= {v_t}/{v}$ is strictly positive.
This in  particular implies that $v_t >0$. Hence, we have the bound
\begin{equation}\label{eqn-bv}
v(\cdot,t) \leq C(t_0)\qquad \mbox{on}\,\,\, (-\infty,t_0)
\end{equation}
for any $t_0 <0$. 
Define the backward limit   
$$v_\infty:=\lim_{t \to -\infty} v(\cdot,t)$$ 
which exists because of the inequality $v_t >0$ but  may vanish at some
points on $S^2$. This actually happens in  our model, the King-Rosenau solution.
As a result,  the  equation \eqref{eqn-p1} fails to be uniformly parabolic  near those
points,  as $t \to -\infty$,  and the standard parabolic and elliptic derivative estimates
fail as well. Nevertheless, it is essential for our
classification result, to establish a priori derivative estimates  which hold uniformly
in time, as $t \to -\infty$. 
 
We recall that  on our ancient solution,  the 
Harnack estimate for the scalar curvature shown in \cite{Ha1} takes the form
$$R_t \ge \frac{|\nabla R|_g^2}{R}$$
which in particular implies that $R_t \geq 0$. Hence, we also have
\begin{equation}\label{eqn-bR}
R(\cdot,t) \leq C\qquad \mbox{on}\,\,\, (-\infty,t_0]
\end{equation}
for any $t_0 <0$, because $R(\cdot,t) \leq R(\cdot,t_0)$.
Once we have the uniform curvature bound along the flow, Shi's derivative estimates  in \cite{Shi} (in the compact case the curvature derivative estimates have been obtained by Hamilton in \cite{Ha4} as well) imply that
$$|\nabla^k R| \le C(k) \,\,\, t \le t_0 < 0.$$
 In addition, the limit
$$R_\infty:= \lim_{t \to -\infty} R(\cdot,t)$$
exists and it is bounded. We will actually show in the next section that $R_\infty=0$ a.e. on $S^2$. 

It is well known   $R$ evolves  by 
$$R_t= \Delta_g R + R^2.$$ 
Expressing  $g = v^{-1}\, g_{S^2}$, in which case $\Delta_g = v \, \Delta_{S^2}$
and $|\nabla\cdot|^2_g = v|\nabla\cdot|_{S^2}^2$, we may  rewrite the Harnack estimate
for $R$ as
$$v\, \dds R + R^2 \ge \frac{v \, |\ds R|^2}{R}$$
or,  equivalently 
\begin{equation}
\label{eq-harnack-R}
\dds R + \frac{R^2}{v} \ge \frac{|\ds R|^2}{R}.
\end{equation}
The pressure $v$ satisfies the elliptic equation
\begin{equation}
\label{eq-ellip-v}
v\, \dds v - |\ds v|^2 + 2v^2 = R\, v.
\end{equation}
We will next use   this equation and the bounds \eqref{eqn-bv}, \eqref{eqn-bR}
and \eqref{eq-harnack-R} to establish uniform first and second order derivative 
estimates on $v$.

\begin{lem}
\label{lem-first-sec-der}
There exists a uniform constant $C$, independent of time, so that
$$\sup_{S^2} \left (|\dds v| +  \frac{|\ds v|^2}{v} \right )(\cdot,t)  \le C 
\qquad 
\mbox{for all} \,\,\,  t\le t_0 < 0.$$
\end{lem}

\begin{proof}
To simplify the notation we will set, throughout the proof,  $\Delta := \dds$ and $\nabla := \nabla_{S^2}$.
We first differentiate 
(\ref{eq-ellip-v}) twice to  compute the equation for $\Delta v$. After a  direct computation, where we also use  the Bochner formula on $S^2$,  we find that 
$$\Delta(\Delta v) + \frac{(\Delta v)^2 - 2\, |\nabla^2 v|^2}{v} + 4\Delta v 
+ 2 \frac{|\nabla v|^2}{v} = \frac{\Delta(Rv)}{v}$$
which implies the inequality  
\begin{equation}
\label{eq-step1}
\Delta(\Delta v + 4v) \ge  -2\, \frac{|\nabla v|^2}{v} + \Delta R + 
2\frac{\nabla R\cdot\nabla v}{v} + \frac{R\, \Delta v}{v}
\end{equation}
since by the trace formula we have 
$$(\Delta v)^2 \le 2|\nabla^2 v|^2.$$
By  (\ref{eq-ellip-v}) we also have 
\begin{equation}
\label{eq-vv}
\Delta v = \frac{|\nabla v|^2}{v} - 2v + R
\end{equation}
and if use this  to replace $\Delta v$ from   the last term
on the right hand side of (\ref{eq-step1}),  we obtain
\begin{equation*}
\begin{split}
\Delta(\Delta v + 4v )  &\ge  -2\frac{|\nabla v|^2}{v} 
+ \Delta R + 2\frac{\nabla R\cdot\nabla v}{v} +
\frac{R}{v} \left (\frac{|\nabla v|^2}{v} - 2v + R \right  ) \\
&= -2 \frac{|\nabla v|^2}{v} + (\Delta R + \frac{R^2}{v})
+  2\frac{\nabla R\cdot\nabla v}{v} + \frac{R \, |\nabla v|^2}{v^2}
- 2R .
\end{split}
\end{equation*}
Combining the last inequality and  the Harnack estimate (\ref{eq-harnack-R}) we obtain  
\begin{equation*}
\begin{split}
\Delta(\Delta v &+ 4v) \ge \,\, -2\frac{|\nabla v|^2}{v} + 
\frac{|\nabla R|^2}{R} + 2\frac{\nabla R \cdot\nabla v}{v} + 
\frac{R|\nabla v|^2}{v^2} - 2R \nonumber\\
&= \,\,-2\frac{|\nabla v|^2}{v} + \frac{1}{R}\left ( |\nabla R|^2 +
2\, \nabla R\cdot R \, \frac{\nabla v}{v} + R^2 \, \frac{|\nabla v|^2}{v^2} \right ) - 2R \nonumber \\
&= \,\, -2\frac{|\nabla v|^2}{v} + \frac{1}{R}\, \left  |\nabla R + 
R\frac{\nabla v}{v} \right |^2 -2R.
\end{split} 
\end{equation*}
Since $R >0$ we conclude the estimate
\begin{equation}\label{eq-lap}
\Delta(\Delta v + 4v) \ge  -2\, \frac{|\nabla v|^2}{v} - 2R.
\end{equation}
If we multiply (\ref{eq-vv}) by $m=2$ and add 
it to (\ref{eq-lap}) we get
$$\Delta(\Delta v + 6\, v) \ge -4v  \geq -C$$
for a uniform constant $C$ (independent of time). By (\ref{eq-vv})
we also have
\begin{equation}
\label{eq-lower-lap}
\Delta v \ge -2v + R > -\bar{C}
\end{equation}
and therefore  
$$X:= \Delta v + \bar{C} + 6v > 0$$
and 
\begin{equation}
\label{eq-moser}
\Delta X \ge -C. 
\end{equation}
Standard Moser iteration applied to (\ref{eq-moser})
yields to the bound 
\begin{equation}
\label{eq-upper-lap}
\sup X \le C_1\, \int_{S^2} X\, da + C_2.
\end{equation}
Observe that
$$\int_{S^2}X\, da = \int_{S^2}(\Delta v + \bar{C} + 6v)\, da 
= \int_{S^2}(\bar{C} + 6v)\, da \le C.$$
The last estimate combined with (\ref{eq-upper-lap}) yields  to the bound
$$\Delta v \le  C.$$
This together with (\ref{eq-lower-lap}) imply
\begin{equation}
\label{eq-lap-bound}
\sup_{S^2}|\Delta v|(\cdot,t) \le C\qquad \mbox{for all} \,\,\, t\in (-\infty,t_0]
\end{equation}
for $C$ is a uniform constant. Since
$$\frac{|\nabla v|^2}{v} = \Delta v + 2v - R$$
the estimate (\ref{eq-lap-bound}) and  $R > 0$, readily imply the bound
\begin{equation}
\label{eq-first-der2}
\sup_{S^2} \, \frac{|\nabla v|^2}{v}  \le C\qquad  \mbox{for all} \,\,\, t\in (-\infty,t_0].
\end{equation}
\end{proof}

As a consequence of the previous lemma  we have:

\begin{cor}\label{cor-c1a}
For any $p \geq 1$, we have
\begin{equation}\label{eqn-w2p}
\| v (\cdot,t)  \|_{W^{2,p}(S^2)} \leq C(p)\qquad \mbox{for all} \,\,\, t\in (-\infty,t_0].
\end{equation}
It follows that  for  any $\alpha < 1$, we have
\begin{equation}\label{eqn-c1a}
\| v (\cdot,t)  \|_{C^{1,\alpha}(S^2)} \leq C(\alpha)\qquad \mbox{for all} \,\,\, t\in (-\infty,t_0].
\end{equation}
\end{cor}
\begin{proof} 
Since $\dds v = f$ in $S^2$, with $f \in L^\infty$, standard $W^{2,p}$
estimates for Laplace's equation imply that $v \in W^{2,p}(S^2)$ for all $ p \geq 1$.
Hence,  \eqref{eqn-c1a} follows by  the Sobolev embedding theorem. 
\end{proof}

We will now use the estimates proven above to improve the regularity of the function $v$. 

\begin{lem}
\label{lem-first-alpha}
For every  $0 < \alpha <1$, there is a uniform constant $C(\alpha)$ so that 
\begin{equation}\label{eqn-first-alpha}
\| |\ds v (\cdot,t)|^2  \|_{C^{1,\alpha}(S^2)} \le C(\alpha) \qquad \mbox{for all} \,\,\, t \le t_0 < 0.
\end{equation}
\end{lem}

\begin{proof}
To simplify the notation we will set $\Delta := \dds$ and $\nabla := \nabla_{S^2}$.
A direct computation shows that $|\nabla v|^2$ satisfies the evolution equation
$$\frac{\partial}{\partial t}\,  |\nabla v|^2 = v\, \Delta (|\nabla v|^2) -
6v\, |\nabla^2v|^2 + 2v|\nabla v|^2 + 2|\nabla v|^2\, \Delta v -
2 \, \nabla(|\nabla v|^2) \cdot \nabla v.$$
On the other hand, differentiating  the equation $v_t=R\, v$ gives 
$$\frac{\partial}{\partial t} \,  |\nabla v|^2= 2 \, \nabla(Rv) \cdot \nabla v.$$
Combining the above yields to 
\begin{equation}
\label{eq-first-der-ellip}
\Delta (|\nabla v|^2)  = f
\end{equation}
with $f$ given by 
\begin{equation}
\label{eqn-def-f}
f=2|\nabla^2v|^2 - 6 \, |\nabla v|^2 - 2\frac{|\nabla
v|^2}{v}\, \Delta v + \frac{2 \nabla(|\nabla v|^2) \cdot \nabla
v}{v} + \frac{2\nabla(Rv) \cdot \nabla v}{v}.
\end{equation}

\smallskip
We will show that  for every $p \geq 1$, we have 
\begin{equation}\label{eqn-fff}
\|f(\cdot, t)\|_{L^p(S^2)} \leq C(p)  \qquad  \mbox{for all} \,\,\, t\in (-\infty,t_0]
\end{equation}
with $C(p)$ independent of $t$. We will denote in the sequel by $C(p)$ various constants that are
independent of $t$. 
We begin by recalling that by  \eqref{eqn-w2p}, we have 
$$\|\nabla^2v(\cdot,t)\|_{L^p} \le C(p)\qquad  \mbox{for all} \,\,\, t\in (-\infty,t_0].$$
Also, by  Lemma \ref{lem-first-sec-der}, we have
$$\|2\frac{|\nabla v|^2}{v}\, \Delta v \|_{L^p(S^2)} + \| |\nabla v|^2 \|_{L^p(S^2)}
\le C(p)  \qquad  \mbox{for all} \,\,\, t\in (-\infty,t_0].$$
Since
$$|\frac{\nabla (|\nabla v|^2) \cdot \nabla v}{v}| \le 2|\nabla^2 v| \,  \frac{|\nabla v|^2}{v}$$
by the previous estimates,  we have
$$ \| \frac{\nabla|(\nabla v|^2) \cdot \nabla v}{v} \|_{L^p(S^2)}
\le C(p)  \qquad  \mbox{for all} \,\,\, t \le t_0 < 0.$$
We also have
\begin{eqnarray*}
|\frac{ \nabla(Rv) \cdot \nabla v }{v}| &\le& R\, \frac{|\nabla v|^2}{v} + |\nabla R|\, |\nabla v| \\
&\le& C + (\sqrt{v}|\nabla R|) \, \left (\frac{|\nabla v|}{\sqrt{v}} \right ) \le C
\end{eqnarray*}
for all $t \le t_0 < 0$, since $\sqrt{v}\, |\nabla R| = |\nabla R|_{g(t)} \le C$.
We can now conclude that \eqref{eqn-fff} holds,   for $p\ge 1$. Standard elliptic regularity estimates 
applied to \eqref{eq-first-der-ellip}  imply the bound 
$$\| |\nabla v|^2 \|_{W^{2,p}} \le C(p)\qquad \mbox{for all} \,\,\, t\in (-\infty,t_0].$$
Since the previous estimate holds for any $p \ge 1$, by the Sobolev embedding
theorem, we conclude \eqref{eqn-first-alpha}.
\end{proof}

\begin{lem}\label{lem-reg2}
For any $0 < \alpha <1$, there is a uniform in time constant $C(\alpha)$, so that
$$ \|\sqrt{v}\, \ds^2 v \|_{C^{0,\alpha}(S^2)} \le C(\alpha)\qquad \mbox{for all} \,\,\, t \le t_0 < 0. $$
Moreover, 
$$\|v \,  \ds^3 v \|_{L^\infty(S^2)}  \leq C\qquad \mbox{for all} \,\,\, t \le t_0 < 0$$
for a uniform in time constant $C$. 
\end{lem}

\begin{proof}

To simplify the notation we will set $\Delta := \dds$ and $\nabla := \nabla_{S^2}$. To prove the estimate
on $ \|\sqrt{v}\, \nabla^2 v \|_{C^{0,\alpha}(S^2)}$ we observe that we can rewrite (\ref{eq-ellip-v})
in the  form
\begin{equation}
\label{eq-other-form}
\Delta v^{3/2} = \frac{9|\nabla v|^2}{4\sqrt{v}} - 3v^{3/2} + \frac{3}{2}\, R\sqrt{v}.
\end{equation}
We claim that the right hand side of the previous identity has uniformly in time 
bounded $C^{0,\alpha}$ norm,  for any $0 < \alpha < 1$. To see that, observe  that for every $p \geq 1$
and for any $i,j$,  we have
\begin{equation}
\label{eq-step}
\begin{split}
\|\nabla(\frac{\nabla_i v \nabla_j v}{\sqrt{v}})\|_{L^p(S^2)}  &\le C\, \big ( 
\|\nabla^2v\|_{L^p(S^2)} \, \|\frac{|\nabla
v|}{\sqrt{v}}\|_{L^\infty(S^2)} + \| \frac{|\nabla v|^2}{v}
\|_{L^\infty(S^2)}^{3/2} \big ) \\  &\le C(p)
\end{split}
\end{equation}
and also
$$\|\nabla(R\sqrt{v})\|_ {L^\infty(S^2)} \le C$$
since
$$|\nabla(R\sqrt{v})|  \leq  | \nabla R|\, \sqrt{v} + R\, \frac{|\nabla v|}{2\sqrt{v}}
\le |\nabla R|_{g(t)} + C \le \tilde{C}.$$
All of the above inequalities hold uniformly on $t \leq  t_0 <0$. 
By the Sobolev embedding theorem we conclude that   the right hand side of (\ref{eq-other-form}) has uniformly bounded $C^{0,\alpha}$ norm, for any $\alpha <1$. Standard elliptic regularity theory applied to (\ref{eq-other-form}) implies that 
$$\|v^{3/2}\|_{C^{2,\alpha}(S^2)} \le C(\alpha)$$
which in particular yields to the estimate
$$\|\sqrt{v}\, \nabla^2 v \|_{C^{\alpha}} \le C(\alpha)\qquad \mbox{for all} \,\,\, t \in (-\infty,t_0]$$
since
$$\nabla^2_{ij}  v^{3/2} = \frac{3}{4}\, \frac{\nabla_i v \nabla_j v}{\sqrt{v}} + \frac{3}{2}\sqrt{v}\, \nabla^2_{ij} v$$
and the first term on the right hand side is in $C^{0,\alpha}$ by (\ref{eq-step}).
\medskip

To prove the second estimate, we now    rewrite (\ref{eq-ellip-v})  as
\begin{equation}
\label{eq-square}
\Delta v^2 = 4|\nabla v|^2 - 4v^2 + 2R\, v.
\end{equation}
Lemma  \ref{lem-first-alpha} implies that  
$4\, |\nabla v|^2 - 4v^2$ has uniformly in time bounded $C^{1,\alpha}$
norm. We claim the same is true for the term $Rv$. To see that,  we differentiate it  twice and use 
the inequality
$$|\nabla^2 (R\, v)| \le |\nabla^2 v|\,  R + |\nabla^2 R|\, v +
2|\nabla R|\, |\nabla v|.$$
By Lemmas  \ref{lem-first-sec-der} and \ref{lem-first-alpha} and the bounds 
$$ v\,  |\nabla^2 R| = |\nabla^2 R|_g \le C, \quad \sqrt{v}\,  |\nabla  R| = |\nabla  R|_g \le C, \quad R \leq C$$
we conclude that for all $p \geq 1$, we have 
$$\|\nabla^2 (R\, v)\|_{L^p(S^2)} \leq C(p),  \qquad \mbox{for all}  \,\,\, t \in (-\infty,t_0].$$
The Sobolev embedding theorem now implies that 
$\|R\, v\|_{C^{1,\alpha}(S^2)}$  is uniformly bounded in time, for every $\alpha < 1$. 
Standard elliptic theory applied to (\ref{eq-square}) yields to the bound
$\|v^2\|_{C^{3,\alpha}} \le C(\al)$, for all $ t\le t_0 < 0.$ In particular,
$$ \|\nabla^3 v^2 \|_{L^\infty(S_2)} \leq C$$
for  a uniform constant $C$. 
Since
$$\| v \, \nabla^3 v \|_{L^\infty(S_2)} \leq C\, \big ( \|\nabla^3 v^2 \|_{L^\infty(S_2)} +  
\| \sqrt {v} \, \nabla^2 v\|_{L^\infty(S_2)}\, \|\frac {\nabla v}{\sqrt{v}} \|_{L^\infty(S_2)} \big )$$
this readily  implies  the bound
$\|v\, \nabla^3 v \|_{L^\infty(S^2)} \le C.$

 \end{proof}


\section{Lyapunov functional and convergence}\label{sec-Lyapunov}

We introduce, in this section,  the Lyapunov functional
\begin{equation}\label{eqn-J}
J(t) = \ins \left (\frac{|\ds v|^2}{v} - 4\, v \right ) \, d a. 
\end{equation}
We will show next that $J(t)$ is non-decreasing and bounded. In the sequel, we will combine  these
properties with the a priori estimates shown in  the previous section  to 
show   that $v(\cdot,t)$ converges, as $t \to - \infty$,  in the $C^{1,\alpha}$ norm, to
a steady state solution $v_\infty$. 

\begin{lem}
\label{lem-lyapunov}
The Lyapunov functional $J(t)$ is monotone under (\ref{eqn-p1}) and in particular, we have 
\begin{equation}\label{eqn-lyap2}
\frac d{dt} J(t)  = - 2 \, \ins \frac{v_t^2}{v^2}  \,  da  -   \ins  \frac{| \ds v|^2}{v^2}   \, v_t \, da
\end{equation}
\end{lem}

\begin{proof}
To simplify the notation we will set $\Delta := \dds$ and $\nabla := \nabla_{S^2}$.
Equation \eqref{eqn-p1} and a direct calculation show:
\begin{eqnarray*}
\frac d{dt} \ins \frac{|\nabla v|^2}{v} \, da &=& 2\, \ins   \frac{\nabla v \, \nabla v_t }{v} \, da - \ins   \frac{|\nabla v|^2}{v^2}\, v_t \, da\\
&=&
- 2\, \ins \frac{\Delta v}{v} \,  v_t \, da + \ins   \frac{|\nabla v|^2}{v^2}\, v_t \, da\\
&=& - 2 \, \ins ( \frac{v_t }{v^2} + \frac{|\nabla v|^2}{v^2} - 2)  \,  v_t \, da
+ \ins   \frac{|\nabla v|^2}{v^2}\, v_t \, da \\
&=& - 2 \, \ins \frac{v_t^2}{v^2}  \,  da  -   \ins  \frac{| \nabla v|^2}{v^2}   \, v_t \, da + 4\, \ins  v_t \, da.
\end{eqnarray*}
We then conclude that
\begin{equation}\label{eqn-lyap1}
\frac d{dt} \ins \left (\frac{|\nabla v|^2}{v}-  4 v \right ) \, d a = - 2 \, \ins \frac{v_t^2}{v^2}  \,  
da  -   \ins  \frac{| \nabla v|^2}{v^2}   \, v_t \, da
\end{equation}
that is, 
\begin{equation}\label{eqn-lyap3}
\frac d{dt} J(t)  = - 2 \, \ins \frac{v_t^2}{v^2}  \,  da  -   \ins  \frac{| \nabla v|^2}{v^2}   \, v_t \, da
\end{equation}
where both terms on the right hand side of \eqref{eqn-lyap3} are
nonpositive, since  on an ancient solution of \eqref{eqn-p1} we have 
$v_t \geq 0$.

\end{proof}

As an immediate consequence of the estimate in Lemma \ref{lem-first-sec-der}
and the inequality $$v\, \Delta v - |\nabla v|^2 + 2\, v^2 \geq 0$$   we have:

\begin{lem}
\label{lem-lyap-bound}
There exists a uniform constant $C$ so that 
$$-C \le J(t) \le 0\qquad \mbox{for all} \,\,\, -\infty < t \leq t_0 <0.$$
\end{lem}
\medskip

We will next use Lemma \ref{lem-lyapunov} to show that on our ancient solution
the backward in time limit 
$$R_\infty= \lim_{t \to -\infty} R(\cdot, t)$$
of the scalar curvature $R$ is equal to zero almost everywhere on $S^2$.

\begin{lem}\label{lem-R} On an ancient solution $v$ of equation \eqref{eqn-p1}, 
we have $R_\infty =0$  a.e. on  $S^2$. 
\end{lem}
\begin{proof} 
It is   enough to show that $$\ins R_\infty^2  \, da =0.$$
Indeed, assume the opposite, namely that $\ins R_\infty^2  \, da := c >0$.
Then, since $R_t \geq 0$,  we  have $\ins  R^2(\cdot,t)  \, da \geq c$,
i.e., 
$$\ins \frac{v_t^2}{v^2}\, da \geq c.$$
Integrating \eqref{eqn-lyap2} in time while  using the above inequality and the fact that
$v_t \geq 0$, gives 
$$J(t_2) - J(t_1) \leq - \int_{t_1}^{t_2} \ins \frac{v_t^2}{v^2}\, da \, dt \leq - c\, (t_2 - t_1)$$
for every $-\infty < t_1 < t_2 < t_0<0$. This obviously contradicts the uniform bound
$-C \leq J(t) \leq 0$  shown in Lemma \ref{lem-lyap-bound}.
\end{proof}

We will now combine some of the a priori estimates of the previous section with Lemma \ref{lem-R}
to prove the following convergence  result. 

\begin{prop}\label{prop-conv} The solution $v(\cdot,t)$ of \eqref{eqn-p1} converges, as $t \to -\infty$, 
to a limit $v_\infty \in C^{1,\alpha}(S^2)$, for any $\alpha <1$. Moreover,  $\|v_\infty \, \ds^2  v_\infty \|_{C^\alpha(S^2)} 
< \infty$, for all $\alpha <1$,  and $v_\infty$  satisfies the steady state equation
\begin{equation}\label{eqn-tildev}
v_\infty\, \dds v_\infty - |\ds v_\infty|^2 + 2\, v_\infty^2 =0.
\end{equation}
\end{prop}

\begin{proof}
Since $v_t \ge 0$ and $v > 0$, the  pointwise limit 
$$v_\infty:=\lim_{t\to -\infty}v(\cdot,t)$$
exists.  
Lemmas \ref{lem-first-alpha} and \ref{lem-reg2}
ensure  that for every $\al <1$ and every sequence $t_i\to -\infty$,  along a subsequence still denoted by $t_i$, we have $v(\cdot,t_i)\stackrel{C^{1,\alpha}(S^2)}{\longrightarrow}
\tilde {v}$ and $v\, \ds^2 v(\cdot,t_i) \stackrel{C^{\alpha}(S^2)}{\longrightarrow}
\tilde {v}  \ds^2 \tilde  v$. By the uniqueness of the limit, $\tilde{v} =
v_\infty$, which means that for every $\al <1$, we have
\begin{equation*}
\label{eq-un-lim}
v(\cdot,t) \stackrel{C^{1,\alpha}(S^2)}{\longrightarrow}
v_\infty \quad \mbox{and} \quad  (v\, \ds^2 v)(\cdot,t) \stackrel{C^{\alpha}(S^2)}{\longrightarrow}
v_\infty  \ds^2 v_\infty, \quad  \mbox{as} \,\,\, t\to -\infty.
\end{equation*}
We can now let $t \to -\infty$ in equation
$$v\, \dds v - |\ds v|^2 + 2\, v^2 = R\, v$$
and use Lemma \ref{lem-R} to conclude that $v_\infty$ satisfies equation \eqref{eqn-tildev}. 

\end{proof}


\section{The backward limit}
\label{sec-limit}

In this section we will  classify all the backward limits  $v_\infty=\lim_{t \to -\infty} v(\cdot,t)$ proving
\begin{thm}
\label{prop-limit}
There exists a conformal change of $S^2$
in which the limit 
$$v_\infty(\psi,\theta):= \lim_{t\to -\infty} v(\psi,\theta,t) = \mu \,  \cos^2\psi$$
for some constant $\mu  \ge 0$, where $\psi,\theta$ denote  global
coordinates on a conformally changed sphere. Moreover, the convergence
is in $C^{1,\alpha}$ on $S^2$, for any $0 < \alpha <1$,  and in  $C^\infty$  on every 
compact subset of $S^2\backslash \{S,N\}$, where $S,N$ denote  the south, north poles of  $S^2$ respectively  (the points that correspond to $\psi = \pm
\frac{\pi}{2}$).
\end{thm}

We have shown in the previous section that  $v(\cdot,t)\stackrel{C^{1,\alpha}(S^2)}{\longrightarrow} v_\infty$,
for any $\alpha \in (0,1)$,  where $v_\infty$ is a weak solution of  the steady state equation
$$v_\infty\, \Delta v_\infty - |\nabla v_\infty|^2 + 2\, v_\infty^2 = 0.$$
To classify  all the backward limits $v_\infty$, we will need the following Proposition  which constitutes
the main step in the proof of Theorem \ref{prop-limit}.

\begin{prop} \label{lem-zero}
The limit $v_\infty$ is either identically equal to zero,  or it has at most   two zeros.
\end{prop}

For a fixed $t_0 <0$, the conformal factor $u$ of our evolving metric on $S^2$ is uniformly bounded from above
and below away from zero. Set   $$m(t_0):=\inf_{z \in S^2} u(z,t_0).$$

Assume that  $v_\infty$ is not  identically equal to zero. Then, since $v_\infty$ is a continuous 
function,   there exist two  points $P_1, P_2 \in S^2$ such that 
$\lim_{t \to \infty} v(P_i,t) >0$, $i=1,2$, or equivalently  $\lim_{t \to \infty} u(P_1,t) <\infty$,
$i=1,2$. By performing a conformal change of coordinates, we may assume that 
$P_1$ is the  south pole $S$ and $P_2$ is   the north pole $N$ of the background sphere 
$S^2$.  Let
$\psi, \theta$ be global coordinates on $S^2$, where $\psi =
\frac{\pi}{2}$ and $\psi = -\frac{\pi}{2}$ correspond to the poles
(denote them by $S$ and $N$). Denote by 
$\bar u$ the conformal factor  of our metric
in plane coordinates, after the stereographic  projection that maps $S$  to the origin. It follows that  
$$\lim_{t \to -\infty}  \bar u (0,t) = \bar u_\infty  (0)< \infty.$$
We have seen that $\bar u$ satisfies equation \eqref{eqn-baru}.   We will show:

\begin{lem}\label{lem-l1b}  Given any $r_0 >0$ and $t_0 <0$, there exists a uniform in time constant $C(r_0)$ which  also depends on 
$\bar u_\infty  (0)$ and $m(t_0)$  such that
\begin{equation}\label{eqn-bl1}
\int_{|x| \leq r_0}  (\log \bar u)^+(x,t)\, dx  \leq C(r_0)\qquad \mbox{for all} \,\,\, t\leq t_0.
\end{equation}
\end{lem}

\begin{proof} 
Set 
$${\bar  U}(r,t) =  \int_0^{2\pi}\log {\bar  u}(r,\theta,t)\, d\theta.$$
Since 
$\Delta  \log {\bar  u}  =  -R\, {\bar u} \le 0$, 
integrating  the inequality $\Delta  \log {\bar  u} \leq 0$ in $\theta$  yields to the inequality
$$r^{-1} \big ( r \, \bar U_r)_r \leq 0.$$
Integration in $r$ (using that $\bar u(\cdot, t)  $ is smooth at the origin) shows  that  $U_r \leq 0$,
implying  the bound
$$\bar U(r,t) \leq  \log\bar u(0,t) \leq \log \bar u_\infty(0) < \infty.
$$
Hence, for any $t \leq t_0$, we have 
$$
\int_{|x| \leq r_0}  \log \bar u(x,t)\, dx  \leq C\, \log \bar u_\infty(0)\, r_0^2 \leq C_1(r_0)
$$
for a uniform in $t$ constant  $C_1(r_0)$. In addition, the inequality $u_t \leq 0$ implies the bound
$$\inf_{ x  \in B_{r_0}(0) }   \bar u(x,t)  \geq \inf_{ x  \in B_{r_0}(0) }  \bar u(x,t_0) \geq c (m(t_0),r_0) >0$$
which gives  
$$\int_{|x| \leq r_0} ( \log \bar u )^{-} (x,t) \, dx  \leq   C_2(r_0).
$$
Combining the previous  two integral bounds yields   \eqref{eqn-bl1}. 
\end{proof}

The following $L^\infty$ bound   is inspired by the beautiful paper of Brezis and Merle \cite{BM}. 
It will play a crucial role in the proof of  Proposition \ref{lem-zero}. 

\begin{lem}\label{lem-bm}  Let $\delta >0$ be a given small number.  If for some $t \leq t_0$, $\rho <1$  and $x_0 \in R^2$, with $|x_0| \leq r_0$,  we have 
\begin{equation}\label{eqn-ib1}
\int_{B_\rho(x_0)} R \, \bar u (x,t) \, dx \leq 4\pi -2 \delta, 
\end{equation}
then
\begin{equation}\label{eqn-gb1}
\sup_{B_{\rho/4} (x_0)} u (\cdot,t) \leq C(r_0, \rho, \delta)
\end{equation}
for a constant $C(r_0, \rho, \delta)$ which depends on $r_0$, $\rho$, $\delta$, $\bar u_\infty(0)$ and $m(t_0)$
but is independent of  time $t$.  

\end{lem}

The proof of the  bound \eqref{eqn-gb1} will
use the ideas  of Brezis and Merle \cite{BM}, 
including the following result which we state for the reader's convenience. 

\begin{thm}[Brezis-Merle]
\label{thm-bm}
Assume that  $\Omega \subset  \mathrm{R}^2$ is a bounded domain and let $w$ be a 
solution of 
\begin{equation}\label{eqn-ell}
\begin{cases}
-\Delta w = f(x) \qquad &\mbox{in} \,\,\, \Omega, \cr
\,\,\,\, \quad u = 0 \qquad &\mbox{on} \,\,\, \partial\Omega,
\end{cases}
\end{equation}
with $f\in L^1(\Omega)$. Then,  for every $\delta \in (0,4\pi)$ we have
$$\int_{\Omega}e^{\frac{(4\pi-\delta)|w(x)|}{\|f\|_{L^1(\Omega)}}} \, dx
\le \frac{4\pi^2}{\delta}(\diam\Omega)^2.$$
\end{thm}

\begin{proof}[Proof of Lemma \ref{lem-bm}]
Fix $t \leq t_0$  so that \eqref{eqn-ib1}
holds, according to the statement of the lemma. Throughout the proof of the lemma we will denote by  $C(\rho,\delta)$ various constants which depend on 
 $\rho$ and $\delta$ but are independent of time $t$.  

Set   $w :=\log \bar u (\cdot, t) $ and observe that $w$ solves the elliptic equation 
$$-\Delta w = R \, e^w\qquad \mbox{in} \,\,\, B_\rho(x_0)$$
with $R$ denoting the scalar curvature.
Let $w_1$
solve problem \eqref{eqn-ell} in $\Omega:=B_\rho(x_0)$ with $f := R\, e^w$. Since
$$\| f \|_{L^1(\Omega)} \leq 4\pi - 2\delta$$
by our assumption \eqref{eqn-ib1}, Theorem \ref{thm-bm} implies the bound
\begin{equation}\label{eqn-bbm} 
\int_{B_\rho(x_0) } e^{p \, |w_1(x)|} \, dx
\le C(\rho,\delta)
\end{equation}
with
$$p:= \frac {4\pi-\delta}{4\pi - 2\delta} >1.$$
Combining   \eqref{eqn-bbm} and Jensen's inequality gives the estimate 
\begin{equation}
\label{eqn-w1}
\|w_1 \|_{L^1(B_\rho(x_0))} \le C(\rho,\delta).
\end{equation}
The difference  $w_2 := w-w_1$ satisfies  $\Delta w_2 = 0$ on $B_\rho(x_0)$. Hence by the  mean value
inequality 
\begin{equation}
\label{eqn-w2}
\|w_2^+\|_{L^{\infty}(B_{\rho/2} (x_0))} \le C(\rho) \, \| w_2^+ \|_{L^1(B_\rho(x_0))}.
\end{equation}
Since
$w_2^+ \le w^+ + |w_1|$ 
combining \eqref{eqn-w1} and \eqref{eqn-bl1}  yields the bound
$$\| w_2^+ \|_{L^1(B_\rho(x_0))} \le C(r_0,\rho,\delta)$$
if $|x_0| \leq r_0$, with $C(r_0,\rho,\delta)$ also depending on $m(t_0)$ and $\bar u_\infty(0)$,
as in the statement of Lemma \ref{lem-l1b}. 
Express   $R\, e^w = R\, e^{w_2}\, e^{w_1}$ and recall  that $-\Delta w = R \, e^w$
with $R$ is uniformly bounded, so that $R \, e^w \le C \, e^{w_1}$
on $B_{\rho/2}(x_0)$ by \eqref{eqn-w2}.  
We conclude by standard elliptic estimates and  \eqref{eqn-bbm} that
$$
\|w^+\|_{L^{\infty}(B_{\rho/4}(x_0))} \le C\,
\big (  \|w^+\|_{L^1(B_{\rho/2}(x_0))} + \|e^{w_1}\|_{L^p(B_{\rho/2}(x_0))}  \big ) 
\le C(\rho,\delta,\rho_0) 
$$
finishing the proof of the Lemma.
\end{proof}

We will now prove Proposition \ref{lem-zero}.

\begin{proof} We argue by contradiction. Assume that there exist at least three distinct points $p_i$, $i=1,2,3$
such that $\lim_{t \to \infty} v(p_i,t) =0$, $i=1,2,3$, or equivalently  $\lim_{t \to \infty} u(p_i,t) =+\infty$.
By our choice of the north and south poles in our given coordinates 
(as in the beginning of this section) these points belong to
$S^2 \setminus \{ S, N \}$, hence they are mapped to three distinct points $x_i$, $i=1,2,3$ on $R^2$ via
the stereographic  projection that maps $S$ to the origin. We choose  $0 < \rho <1$  so that 
all balls $B_\rho(x_i)$ are disjoint.  

Let  $\delta <1/2$ be a given positive number. Given any sequence $t_k \to \infty$ and any of the three points 
$x_i$, we may choose a subsequence, still denoted by $\{ t_k \}$,  such that 
$$
\int_{B_\rho(x_i)} R \, \bar u (x,t_k) \, dx >  4\pi -  2\delta$$
for the particular point $x_i$. Otherwise $\lim_{t \to \infty} u(x_i,t) < \infty$ by \eqref{eqn-gb1} which would contradict the choice of $x_i$. This readily implies the existence of a sequence $t_k \to -\infty$ for which  
\begin{equation}\label{eqn-g1}
\int_{B_\rho(x_i)} R \, \bar u (x,t_k) \, dx >  4\pi - 2 \delta,   \qquad \forall k
\end{equation}
and for all three points $x_i$, $i=1,2,3$. 

Recall that the balls $B_\rho(x_i)$ are chosen to be disjoint.
It follows that the total curvature for our metric $g(t_k) := \bar u(\cdot, t_k) $ satisfies
$$\int_{\R^2} R\, \bar u \, dx \geq \sum_{i=1}^3 \int_{B_\rho(x_i)} R \, \bar u (x,t_k) \, dx >  12 \pi -  6\delta > 8\pi$$
if $\delta < 1/2$, a contradiction to the total curvature of $g(t_k)$ being equal to $8\pi$. This completes the proof of  the proposition.

\end{proof}

We are now in a position to classify all  backward limits  $u_\infty$.

\begin{proof}[Proof of Theorem \ref{prop-limit}] 
Assume from now on that the backward limit $v_\infty$ is not
identically zero.  We have just shown in Proposition \ref{lem-zero} that 
$v_\infty$ has at most   two zeros.  Choose a conformal change of $S^2$
which brings those two zeros to two antipodal poles on $S^2$ (if there is
only one zero, we bring it to the south pole and we choose 
for a north pole its antipodal point).  Let
$\psi, \theta$ be global coordinates on $S^2$, where $\psi =
\frac{\pi}{2}$ and $\psi = -\frac{\pi}{2}$ correspond to the poles
(denote them by $S$ and $N$). Observe that equation \eqref{eqn-p1} is
strictly parabolic away from the poles, uniformly as $t \to -
\infty$. It follows by standard parabolic PDE arguments,  that the convergence $v(\cdot,t) \to v_\infty$,
as $t \to -\infty$, is in $C^\infty$  on compact subsets of
$S^2\backslash\{S,N\}$. Perform the Mercator's transformation \eqref{eqn-uU}, and denote by 
  ${\hat v}(s,\theta,t) = v(\psi,\theta,t)\,
\cosh^2x$ the pressure in cylindrical coordinates.  We conclude that
$$\lim_{t \to -\infty} {\hat v}(s,\theta,t) = {\hat v_\infty}(s,\theta) 
:= v_\infty(\psi,\theta)\, \cosh^2x > 0$$ 
and the convergence is smooth on compact subsets of $\mathbb{R}\times [0,2\pi]$.
The limit $\hat v_\infty$ satisfies
$${\hat v_\infty}\Delta {\hat v_\infty} - |\nabla {\hat v_\infty}|^2 = 0$$ 
or (since ${\hat v_\infty}(s,\theta) > 0$ on $\mathbb{R}\times [0,2\pi]$) 
equivalently,
\begin{equation}
\label{eq-steady-simple}
\Delta_c \log{\hat v_\infty} = 0
\end{equation}
where $\Delta_c$ denotes  the cylindrical laplacian on $\mathbb{R}^2$.  To
finish the proof of Theorem \ref{prop-limit} we need to classify
the solutions of  equation  (\ref{eq-steady-simple}) that come as  limits of
ancient solutions $\hat v(\cdot,t)$.
\smallskip

To this end, set $w :=\log{\hat v_\infty}$ so that $\Delta_c w=0$, by \eqref{eq-steady-simple}.
We can
view $w$ as a harmonic function on $\mathbb{R}^2$, after extending it in the 
$\theta$ direction so that it remains $2\pi$-periodic. In addition,  since ${\hat v_\infty}(s,\theta) = v_\infty(\psi,\theta)\, \cosh^2 s$ and 
$v_\infty = \lim_{t\to -\infty} v(\cdot,t) \le C$,  it follows that there are uniform constants $C_1, C_2$ such  that
\begin{equation}\label{eqn-claimf}
w (s,\theta) \le C_1 + C_2\, |s|.
\end{equation}
Since $w$ is harmonic on $\R^2$ it follows by the mean value formula (via a
standard argument) that the same  bound \eqref{eqn-claimf} holds for $|w|$. 
It is now a well known fact that the only harmonic functions on $\R^2$ with linear growth at infinity  are the linear functions.
Since our function $w$ is periodic in $\theta$, it follows that 
$w (s,\theta) = a_1  + a_2 \, s$, for some constants $a_1,a_2$, and after exponentiating we obtain 
\begin{equation}\label{eqn-form}
{\hat v_\infty}(s,\theta) = \mu \,  e^{\la s}
\end{equation}
for some constants $\mu  \ge 0$ and $\la \in \R$. 
Since we have assumed that the function  $\hat v_\infty$ is not identically zero, 
we have  $\mu  > 0$.  

\smallskip
To finish our argument we need to show that $\lambda =0$. We recall  the estimate $|\nabla_{S^2}  v| \le C\sqrt{v}$,  shown in Lemma \ref{lem-first-sec-der}, or equivalently, $|v_{\psi}| \le C\sqrt{v}$, which in cylindrical coordinates gives the bound 
$$|\hat{v}_s(s,\theta,t) - 2\hat{v}(s,\theta,t)\tanh{s}| \le C\sqrt{\hat{v}(s,\theta,t)}$$
which holds 
for $t \le t_0 < 0$. Taking $t\to -\infty$ we obtain
$$|(\hat{v}_{\infty})_s(s,\theta) - 2 \hat{v}_{\infty}(s,\theta)\, \tanh s| \le C\sqrt{\hat{v}_{\infty}(s,\theta)}$$
or equivalently,
$$\sqrt{\mu}\, |\lambda - 2\, \tanh s| \le Ce^{-\lambda s/2} \qquad \forall s \in (-\infty,+\infty)$$
which is impossible unless either $\lambda = 2$ or $\lambda = 0$. 

In the case that  $\lambda =2$, then  if $\bar u$ is the conformal
factor of our metric $g$ parametrized by the standard plane
and $\bar v := \bar u^{-1}$ the corresponding pressure function, 
 then 
$\lim_{t \to -\infty} \bar v(\cdot,t) = \mu$, on $\R^2 \setminus \{0\}.$
 Since $\mu >0$, we have 
$$\lim_{t \to -\infty} \bar u(x,t) = \gamma:= \mu^{-1} < \infty \qquad \mbox{on} \,\, \R^2 \setminus \{0 \}
$$
which in particular implies that $\bar u(\cdot,t)$ is bounded from above and below away from zero 
on any compact subset $K \subset \R^2$. Standard parabolic PDE arguments imply that  $\bar u(\cdot,t) \to \gamma $, 
as $t \to -\infty$, 
in $C^\infty$ on compact subsets of $\R^2$.  By Lemma \ref{prop-plane} (which will be proven at the
end of Section \ref{section-sphere}) this is impossible. We conclude that $\lambda=0$.

\smallskip

The above discussion yields that  $\hat v_\infty(s,\theta)=\mu$, with $\mu \geq 0$. Going 
back to the   sphere $S^2$,  via Mercator's transformation, we conclude that
$$v_\infty(\psi,\theta):= \lim_{t \to - \infty} v(\psi,\theta,t) =  \mu \, \cos^2\psi.$$
Moreover, the convergence is  smooth  on
compact subsets of $S^2\backslash\{S,N\}$. This finishes the proof of the theorem.

\end{proof}


\section{The King-Rosenau Solutions}
\label{section-KR}

We have shown in the previous section that there exists a parametrization of our evolving metric $g(t)$ on $S^2$, namely
$g(t)=u(\psi,\theta,t)\, ds_p^2$ for which the backward limit of the pressure function $v:=u^{-1}$  satisfies 
$$v_\infty(\psi,\theta):= \lim_{t \to - \infty} v(\psi,\theta,t) =  \mu\, \cos^2\psi$$
with $\mu  \geq 0$. Assuming,  throughout this section, that $\mu >0$   we will show
 
\begin{thm}
\label{prop-King}
If  the backward limit $v_\infty (\psi,\theta)=  \mu\, \cos^2\psi$, with $\mu >0$, then 
$v$ is one of the King-Rosenau solutions \eqref{eqn-King-Rosenau}.
\end{thm}

The case where $\mu=0$ will be treated separately in the last section of the paper. 
Observe, that when $\mu >0$, the metric $g_\infty:= v_\infty^{-1} \, ds_p$ is just the cylindrical metric written on $S^2$. 
By performing  a simple rescaling in $t$ and $v$ we may  assume,   without loss of generality, that $\mu=1$. 

\smallskip

Let $S,N$ denote the north  and south pole of the sphere $S^2$ corresponding to $\psi=-\frac \pi 2$ and $\psi = \frac \pi 2$ respectively. 
Consider the stereographic projections  $\Phi_{S}  : S^2 \setminus \{N\} \to \R^2$ and  $\Phi_{N}  : S^2 \setminus \{S\} \to \R^2$
such that $\Phi_S(S)=0$ and $\Phi_N(N) =0$  and set 
$$\bar v_S(x,y,t)  = v (\Phi^{-1}_S (x,y),t)\qquad \mbox{and} \qquad \bar v_N(\zeta, \xi,t) =  v ( \Phi^{-1}_N (\zeta,\xi),t)$$ 
where $v(\psi,\theta,t)$ denotes  the  pressure function on $S^2$. 
Then, if 
\begin{equation}
\label{eqn-var}
\zeta = \frac x{x^2+y^2} \quad \mbox{and} \quad \xi = \frac y{x^2+y^2}
\end{equation}
we have
\begin{equation}\label{eqn-NS}
\bar v_S(x,y,t) = (x^2+ y^2)^2 \, \bar v_N(\zeta,\xi,t).
\end{equation}

Observe that, after stereographic projection, the pressure function in  the King-Rosenau solutions takes  the form
$$\bar v_S(x,y,t) = b(t) + c(t) (x^2+y^2) + b(t) (x^2+y^2)^2$$
and similarly 
$$\bar v_N(\zeta,\xi,t) = b(t) + c(t) (\zeta^2+\xi^2) + b(t) (\zeta^2+\xi^2)^2$$
where $\lim_{t\to -\infty} c(t)= \mu$ and $\lim_{t \to -\infty} b(t) =0$.  

We consider the quantities 
$$Q_S(x,y,t) := \bar v_S \, \big [ \big ( (\bar v_S)_{xxx}  -  3  (\bar v_S)_{xyy} \big )^2  +  \big ( (\bar v_S)_{yyy}  -  3 (\bar v_S)_{xxy} \big )^2 \big ]$$
and, similarly,
$$Q_N(\zeta,\xi,t) :=  \bar v_S \, \big [ \big ( (\bar v_N)_{\zeta\zeta\zeta}  -   3 (\bar v_N)_{\zeta\xi\xi} \big )^2  +  \big ( (\bar v_N)_{\xi\xi\xi}  -    3
(\bar v_N)_{\zeta\zeta\xi} \big )^2 \big ].$$
Both $Q_N$ and $Q_S$ are identically equal to zero on the King-Rosenau solutions. 
A direct calculation (where we make  use of     \eqref{eqn-var} and \eqref{eqn-NS}) shows that 
\begin{equation}\label{eqn-QNS}
Q_S(x,y,t) = Q_N(\zeta,\xi,t).
\end{equation}
Hence, the quantity
$$Q(\psi,\theta,t) := \begin{cases} Q_S(\Phi_S(\psi,\theta),t)\qquad &(\psi,\theta) \in  S^2 \setminus \{N\}\\
Q_N(\Phi_N(\psi,\theta),t)\qquad &(\psi,\theta) \in  S^2 \setminus \{S\}
\end{cases}
$$
is a well defined and smooth function on $S^2 \times (-\infty,0)$. 

\smallskip

{\em Sketch of proof:} We will show next that $Q(\cdot,t)  \equiv 0$,  for all $t <0$,   by showing that its maximum
is decreasing in time and is equal to zero at $t=-\infty$. Following this we will prove that 
any solution of equation \eqref{eqn-p1} which satisfies $Q(\cdot,t)  \equiv 0$ must be one of  the King-Rosenau solutions yielding  the statement of  Proposition \ref{prop-King}. 

\smallskip

Let $$Q_{\max} (t):= \max_{(\psi,\theta) \in S^2} Q(\psi,\theta,t),\qquad t \in (-\infty,0).$$

\begin{lem}\label{lem-mpQ} The function $Q_{\max} (t)$ is decreasing in $t$.  

\end{lem}

\begin{proof} To show that $Q_{\max} (t)$ is decreasing we will compute the evolution equation of $Q$.
We may assume, without loss of generality,  that $Q_{\max}(t)$ at an instant $t$,   is achieved  on the southern
hemisphere corresponding to $-\pi/2 \leq \psi \leq 0$ so that 
$$Q_{\max}(t) = \sup_{(x,y) \in \R^2} Q_S(x,y,t) = Q_S(x_0,y_0,t)$$
for some point $(x_0,y_0) \in \R^2.$  

To simplify the notation, set $\bar v:= \bar v_S$ and 
$$ A := \bv_{xxx} - 3 \, \bv_{xyy}\qquad B:= \bv_{yyy} - 3\, \bv_{xxy}$$
so that 
$$A_x := \bv_{xxxx} - 3 \, \bv_{xxyy}\qquad  B_y:= \bv_{yyyy} - 3\, \bv_{xxyy}$$
and also set
$$\bar Q(x,y,t):=\frac 12 \, Q_S(x,y,t) = \frac \bv 2\, (A^2 + B^2)$$
and
$$D_1 := 2 \bv \, A_x + \bv_x \, A - 3 \bv_y \, B\qquad D_2 := 2 \bv \, B_y + \bv_y \, B - 3 \bv_x \, A.$$
A direct computation shows that 
$$L\bar Q:= \bar Q_t -  \bar v\, \Delta \bar Q =  - a_1 \, \bar Q_x^2 -   b_1 \, \bar Q_y^2 - a_2 \, \bar Q_x -  b_2 \, \bar Q_y  - C\, \bar Q - 4 R \, \bar Q$$
where $R \geq 0$ denotes the scalar curvature of our metric,  and 
$$a_1 =  \frac 1{B^2}, \quad b_1 =  \frac 1{A^2}, \quad a_2 =  \frac {A}{B^2} \, D_1, \quad  b_2 =  \frac {B}{A^2} \, D_2$$
and
$$C=  \frac 2{\bv} \, \big ( \frac {D_1^2}{4B^2} + \frac {D_2^2}{4A^2} \big ).$$
Observe next that 
$$C \bar Q=  \frac 2{\bv} \, \big ( \frac {D_1^2}{4B^2} + \frac {D_2^2}{4A^2} \big ) \, \bar Q=    \frac {A^2 D_1^2}{4B^2} +  \frac {B^2D_2^2}{4A^2} 
+ \frac 14(D_1^2 + D_2^2)$$
and 
$$a_1 \bar Q_x^2 + a_2 \bar Q_x +  \frac {A^2 D_1^2}{4B^2} = \frac 1{B^2} \, \big (\bar Q_x + \frac{AD_1}{2} \big )^2$$
and similarly
$$b_1 \bar Q_y^2 + b_2 \bar Q_y +  \frac {B^2 D_2^2}{4A^2} = \frac 1{A^2} \, \big (\bar Q_y + \frac{BD_2}{2} \big )^2.$$
Hence,
$$L \bar Q = - \frac 1{B^2} \, \big (\bar Q_x + \frac{AD_1}{2} \big )^2 - \frac 1{A^2} \, \big (\bar Q_y + \frac{BD_2}{2} \big )^2 - \frac 14(D_1^2 + D_2^2) - R\, \bar Q$$
where we recall that $R \geq 0$ everywhere. 
Since $\bar Q$ is smooth (because $\bv$ is) it follows that all quantities on the right hand side of the above equation are bounded 
at any given point  $(x,y,t) \in \R^2 \times (-\infty,0)$ and 
$$LQ:= \bar Q_t -  v\, \Delta \bar Q \leq 0\qquad \mbox{for all} \,\, (x,y,t) \in \R^2 \times (-\infty,0).$$
This readily implies that $Q_{\max}(t)$ is decreasing in $t$, finishing the proof of the lemma.
\end{proof}

We will next show  that the backward limit as $t \to -\infty$ of $Q_{\max} (t)$ is zero.

\begin{lem}\label{lem-limit}  We have $$\lim_{t \to -\infty} Q_{\max} (t) =0.$$
\end{lem}

As above, we  set $\bar v (x,y,t) = \bar v_S(x,y,t)$ and  consider the conformal factor $\bar u  =  \bar v^{-1}$.
Our evolving metric $g(t)$ is then given by $g(t) = \bar u (\cdot,t) \, (dx^2 + dy^2)$,  where $dx^2 + dy^2$ denotes  the standard metric on the plane.
Recall that the  function $\bar u$ satisfies the evolution equation \eqref{eqn-baru}. 

To simplify the notation, we will also denote by $\bar u$  the conformal factor 
of our metric over the plane $\R^2$ expressed in polar coordinates. 
Then, 
$$g(t) = \bar u(\cdot, t) \, (dr^2 + r^2 \, d\theta^2)  = \hat u (\cdot,t) \, (ds^2 + d\theta^2)$$
where $\hat u$ is the conformal factor in cylindrical coordinates,  defined in terms of $u$ by \eqref{eqn-uU}. 

\smallskip

In the proof of Lemma \ref{lem-limit} we will use the following estimate.

\begin{lem}\label{lem-bounds} For  any $t_0 <0$, there exists  a uniform in time constant $C$, 
depending only on  $t_0$,  such that 
\begin{equation}
\label{eqn-bounds-theta}
|(\log \bar u)_\theta(\cdot,t) | = |(\log \hat u)_\theta(\cdot,t) | \leq C\qquad \mbox{on} \,\,\, -\infty < t \leq t_0.
\end{equation} 
In addition, 
$\, \hat u (\cdot,t)\leq 1\, $ and $\, r^2 \, \bar u (\cdot,t)\leq 1\, $, for all $-\infty < t \leq t_0< 0$. 
\end{lem} 
\begin{proof} We have seen in Lemma  \ref{lem-first-sec-der} that the pressure function $v$ written on $S^2$  satisfies the bound
$$|\ds v|^2 \leq C\, v\qquad \mbox{on}\,\,  -\infty < t \leq t_0$$
for a uniform constant $C$. This readily gives us the bound
$$ \sec^2 \psi \, |v_\theta(\psi,\theta,t)|^2 \leq C\, v(\psi,\theta,t)$$
or, equivalently,
$$|v_\theta(\psi,\theta,t)|^2 \leq C\, v(\psi,\theta,t)\, \cos^2 \psi.$$
However, since $v_t \geq 0$, we have $v(\psi,\theta,t) \geq \lim_{t \to -\infty} v(\psi,\theta,t) = \cos^2 \psi$. It follows that
$$|v_\theta(\psi,\theta,t)|\leq C\, v(\psi,\theta,t)$$
or equivalently, $|(\log v)_\theta(\cdot,t) | \leq C$. Hence, the conformal factor $u:=v^{-1}$ also satisfies
$$|(\log u)_\theta(\cdot,t) | \leq C$$
The bounds \eqref{eqn-bounds-theta} now readily follow from \eqref{eqn-uU} and \eqref{eqn-Ubaru}.

For the  $L^\infty$ bounds on $\hat u$ and $\bar u$, we use that  $\hat u_t \leq 0$, which implies   $\hat u (\cdot,t)  \leq \lim_{t \to -\infty} \hat u(\cdot,t)=1$ giving  the  bound $\hat u(\cdot,t) \leq 1$ and  also yielding   that $r^2 \, \bar u(\cdot,t) \leq 1$.  
\end{proof}

For a given sequence $t_k \to -\infty$ we define the re-scaled  solutions of 
\eqref{eqn-baru} given  by 
\begin{equation}\label{eqn-ukkk}
\bar u_k (x,y,t) := \rho_k^2 \,  \bar u(\rho_k x, \rho_k y,t+t_k) 
\end{equation}
where  $\rho_k^2=(\bar u(0,t_k))^{-1}$ chosen so that 
$$\bar u_k (0,0) = 1.$$ Before we give the proof of Lemma  \ref{lem-limit},
we will show

\begin{lem}\label{lem-cigar} Passing to a subsequence,  $\{ \bar u_k \}$ converges  uniformly on compact subsets of $\R \times (-\infty, \infty)$,
to a cigar  solution $\bar {\bar u}$ given by 
\begin{equation}
\label{eqn-cigar}
 \bar {\bar u} (x,y,t) = \frac {\alpha}{\beta e^{2\lambda t} + (x-x_0)^2+(y-y_0)^2}
 \end{equation}
for some constants $\alpha, \beta >0$ and $\lambda$  and some point $(x_0,y_0) \in \R^2$.  
\end{lem}
\begin{proof} It is more convenient to switch for the moment to polar coordinates, defining 
$\bar u_k (r,\theta,t) = \bar u(\rho_k r, \theta,t+t_k) \, \rho_k^2$. 
 We will first show the bounds
\begin{equation}
\label{eqn-buk1}
 - C \, (1+ r^2)  \leq \log \buk (r,\theta,0) \leq C 
 \end{equation}
for a uniform constant $C$ (independent of $k$). To this end, we begin by observing that $\log \buk$ satisfies the elliptic equation
$$\Delta \log \buk = - R_k \, \buk$$
where $R_k(r,\theta,t) = R(\rho_k \, r,\theta,t+t_k)$ satisfies the uniform bound 
$$0 < R_k \leq M.$$
Set 
$${\bar U}_k  (r) = \int_0^{2\pi} \log \buk (r,\theta,0) \, d\theta, \qquad \,\, r \geq 0$$  
and observe that by integrating the inequality 
$$\Delta \log \bar u_k (\cdot,0) \leq 0$$
 in $\theta$ we obtain the differential inequality
$$\Delta {\bar U}_k = r^{-1} ( r\, {\bar U}_k'(r))' \leq 0.$$
Since $\lim_{r \to 0} r\, {\bar U}_k'(r) =0$, we readily conclude that ${\bar U}_k(r)$ is decreasing in $r$, hence 
$$\int_0^{2\pi} \log \buk (r,\theta,0) \, d\theta \leq  \log \buk (0,0)=0.$$
In addition, by \eqref{eqn-bounds-theta} we have $|(\log \buk)_\theta(\cdot,0) | \leq C$, for a uniform constant $C$. The last two inequalities 
clearly imply  the  bound from  above in \eqref{eqn-buk1}.
For the bound from below observe that 
$$- \Delta \log \buk (\cdot,0) = R_k \buk (\cdot,0) \leq C$$
for a uniform constant $C$, which gives (after integration in $\theta$) the differential inequality
$$ - r^{-1} ( r\, {\bar U}_k'(r))' \leq C.$$
The desired bound now readily follows by integrating  in $r$ and using \eqref{eqn-bounds-theta}. This proves 
\eqref{eqn-buk1}. 

Now for a given $\tau >0$, we  choose $k$ sufficiently large so that $t_k + \tau < -1$, hence 
$$ \max_{ \R^2 \times (-\infty ,\tau]}   R_k  \leq \max_{ \R^2 \times (-\infty ,-1]}  R \leq M$$
for a uniform constant $M$. Since $(\log \buk)_t = - R_k$, from \eqref{eqn-buk1} we readily conclude the bounds
$$
 - C(\tau)  \, (1+ r^2)  \leq \log \buk (r,\theta,t) \leq C(\tau)\qquad   \mbox{on} \,\, \R^2 \times [-\tau,\tau]
 $$
for a constant $C(\tau)$ that depends on $\tau$ but is uniform in $k$. Exponentiating gives us the bounds
$$0 < c(\tau, r) \leq \buk (r,\theta,t)   \leq C(\tau) < \infty\qquad   \mbox{on} \,\, \R^2 \times [-\tau,\tau].$$
Standard  parabolic PDE arguments imply that the sequence $\{ \bar u_k \}$ is equicontinuous on compact subsets 
of $\R^2 \times (-\infty,\infty)$, hence,  passing to a subsequence,  $\{ \bar u_k \}$
 converges, uniformly on compact subsets of $\R^2 \times (-\infty, \infty)$,
to an eternal solution $\bbbu$ of equation 
$$\bbbu_t = \Delta \log \bbbu\qquad \mbox{on} \,\, \R^2 \times (-\infty,\infty)$$
which in addition satisfies the bound 
$0 <  \bbbu (\cdot,t)  \leq C(t)$, 
for all $t$.

The result in \cite{DS} now shows that $\bbbu $ is either a  cigar solution  or the constant solution $\bbbu \equiv \alpha$. In the latter case,  given $r  >0$, we may find $k$ sufficiently large (depending on $r$) so that  
$$\buk (r,\theta,0):= \rho_k^2\, u(\rho_k r,  \theta, t_k)   \geq \frac \alpha 2$$
for all $\theta \in [0,2\pi]$. 
It follows that that 
$$(\rho_k r)^2 \,  \bbu(\rho_k r,  \theta, t_k)  \geq \frac {r^2 \alpha}{2}.$$
This will contradict our uniform bound $r^2 \, \bbu(r,t) \leq 1$ shown in Lemma \ref{lem-bounds} if we choose $r^2 =4/a$.  

We conclude that our limit  $\bbbu$ is a cigar solution which in standard plane coordinates $(x,y)$ takes the  form \eqref{eqn-cigar}.
The proof of the lemma  is now complete.
\end{proof}

We will now proceed  to the proof of Lemma \ref{lem-limit}.

\begin{proof}[Proof of Lemma \ref{lem-limit}] We begin by noticing that our quantity $Q(\psi,\theta,t)$ becomes identically equal to zero if
$v$ is either the cigar solution or the cylinder. Hence, the convergence of 
$v(\cdot,t)$ to the cylindrical metric  in $C^\infty(S^2\setminus \{ S,N\})$ readily shows that 
$Q(\cdot,t) $ converges uniformly to zero as $t \to -\infty$ on compact subsets of $S^2\setminus \{ S,N\}$. 

To prove the lemma, we  argue by contradiction. If the conclusion of the lemma  doesn't  hold,  then there
exists a sequence of times $t_k$ and points $P_k \in S^2$ such that 
\begin{equation}
\label{eqn-QPk}
Q(P_k,t_k) \geq \e >0.
\end{equation} 
It follows from  the above discussion that we may assume, without loss of generality, that $P_k \to S$ as $k \to \infty$, where $S$ denotes the
south  pole of the sphere corresponding to $\psi=-\pi/2$ in the chosen coordinates. Denote by $\bar P_k = (r_k,\theta_k)$ the
polar  coordinates of the points $P_k$ on the plane obtained by projecting $S^2\setminus \{ N \}$ onto $\R^2$ and mapping $S$ to the origin.

Set  $\rho_k^2 : = (\bar u(0,t_k))^{-1}$ and let $\buk$ be  the sequence 
of rescaled solutions defined by \eqref{eqn-ukkk} and used in Lemma \ref{lem-cigar}. 
We will separate between the following two cases:

\smallskip

\noindent{\em Case 1: We have $\liminf_{k \to \infty} r_k /\rho_k  < \infty$}. 

\smallskip
In  this case, we may assume without
loss of generality,  that $(\bar r_k,\theta_k) := ( r_k / \rho_k, \theta_k)  \to (r_0,0) $, 
with $r_0 < \infty$  (otherwise we pass to a subsequence and rotate in $\theta$).  Since  
$\buk(\bar r_k,\theta_k, 0) = \rho_k^2 \, \bu (\bar r_k \rho_k,\theta_k, t_k)$,  
the convergence of $\buk$ to the cigar  solution readily implies that 
$$\lim_{k \to \infty} \bar Q_k(\bar r_k, \theta_k,0) =0$$
 (where $\bar Q_k$ is
our given quantity corresponding to $\buk$ when expressed in polar coordinates on the plane). However, since this quantity is dilation
invariant, we have $$\bar Q_k(\bar r_k, \theta_k,0) = Q( r_k, \theta_k, t_k)= Q(P_k, t_k) \geq \e >0$$ 
which contradicts that the limit is zero. 
\smallskip

\noindent{\em Case 2: We have $\liminf_{k \to \infty}  r_k /\rho_k  = +\infty$}.  

\smallskip

It is more convenient to work in cylindrical coordinates and set 
$${\hat u}(s,\theta,t) = r^2 \, \bu (r,\theta,t), \qquad r=e^s$$
recalling that ${\hat u}$ satisfies the equation \eqref{eqn-hatu}. 
We set
$$\hat U(s,t):= \int_0^{2\pi} \log {\hat u}(s,\theta,t)\, d\theta$$
and observe that   that since $\Delta_c \log {\hat u} = \hat u_t \leq 0$,  we have ${\hat U}_{ss}(s,t) \leq 0$, hence ${\hat U}_s$ is non-increasing in $s$. In addition, by \eqref{eqn-uU}, we 
 have
$$\log {\hat u}(s,\theta,t) = \log u(\psi,\theta,t) - 2\log \cosh s$$
hence 
$$\lim_{s\to -\infty} {\hat U}_s(\cdot,t) = 2\qquad \mbox{and} \qquad \lim_{s\to \infty} {\hat U}_s 
(\cdot,t)= -2, \qquad \forall t \in (-\infty,0).$$
It follows that 
$$|{\hat U}_s(\cdot,t)| \leq 2, \qquad \forall t \in (-\infty,0).$$

We claim that if  $s_k:=\log r_k$, we have 
\begin{equation}
\label{eqn-boundF}
\hat U(s_k,t_k) \geq -C 
\end{equation}
for some constant $C>0$.  To this end, choose $\hat r$ sufficiently large so that if $\bbbu(r,\theta, t)$ is the cigar solution given in \eqref{eqn-cigar} expressed 
in polar coordinates, then
$$\hat r^2 \, \bbbu (\hat r,\theta, 0) \geq \frac {2\al}3.$$
This is possible because  $\lim_{r \to +\infty} r^2 \,  \bbbu(r,\theta, 0) = \alpha$. 
Since $\hat r^2 \, \buk (\hat r, \theta, 0) \to \hat r^2 \, \bbbu (\hat r,\theta, 0)$, as $k \to \infty$, we must have 
\begin{equation}
\label{eqn-bu32}
\hat r^2 \rho_k^2 \, \bu (\rho_k \hat r, \theta,t_k)  \geq  \frac{\al}2
\end{equation}
if $k$ is sufficiently large and $\theta \in [0,2\pi]$. It follows that if  $\hat s_k := \log (\hat r  \rho_k)$, then 
$${\hat u}(\hat s_k,\theta_k,t_k) \geq \frac \alpha 2.$$
Since $ r_k /\rho_k \to +\infty$, we may assume that $ \hat r  \rho_k < < r_k  $ which in particular implies that  $\hat s_k < s_k$. 
By \eqref{eqn-bounds-theta}
and the  bound from below on ${\hat u}$, we have ${\hat U}(\hat s_k,t_k) \geq -C$, for a uniform constant $C$. 

We will now conclude that the bound \eqref{eqn-boundF} holds. If ${\hat U}(s_k,t_k) \to 0$, as $k \to \infty$, then it  obviously holds. 
Otherwise,  since  $\lim_{s \to -\infty}  {\hat U}(s,t_k) = -\infty$
and $\lim_{t_k \to -\infty} {\hat U}(s,t_k) =0$ on compact subsets of $\R$ (remember
$ \lim_{t \to -\infty} {\hat u}(s,\theta,t) =\mu$ on compact subsets of $\R\times [0,2\pi]$
and we have assumed that $\mu=1$) we easily conclude that ${\hat U}_s (s,t_k)  \geq 0$ for 
$s \leq s_k$ (recall that  $s_k:=\log r_k  \to -\infty).$ It follows that
${\hat U}(s_k,t_k) \geq {\hat U}(\hat s_k,t_k) \geq -C$, which proves  \eqref{eqn-boundF}. 

\smallskip
For the given sequences $t_k \to -\infty$ and $s_k \to -\infty$, we   define the translating solutions
$${\hat u}_k(s,\theta,t):= {\hat u}(s+s_k,\theta,t+t_k)$$
which also satisfy equation \eqref{eqn-hatu} on $- \infty < t < |t_k|$. 
Set  $${\hat U}_k (s,t):= \int_0^{2\pi} \log {\hat u}_k(s,\theta,t)\, d\theta.$$
Then, $|({\hat U}_k)_s| \leq 2$ and $ {\hat U}_k \leq 0$ on 
$\R \times (-\infty,|t_k|-1)$,  since $|{\hat U}_s| \leq 2$ and ${\hat u} \leq 1$ 
on $\R \times (-\infty,-1)$. 
In addition,  by \eqref{eqn-boundF},
$ {\hat U}_k (0,0) \geq -C$ and also $|({\hat U}_k)_t| \leq C$, for all $s \in R$ and $t < |t_k|-1$  for a uniform constant $C$ (since the scalar curvature  $R(\cdot,t)$ is uniformly bounded on $t < -1$). 

It follows that  the sequence $\{ {\hat U}_k \} $ is uniformly bounded on compact sets in  space and time and by \eqref{eqn-bounds-theta} 
the same holds for the sequence $\log {\hat u}_k $. Hence, for a given compact set $K \subset \R \times [0,2\pi] \times (-\infty,\infty)$,
we have 
$$ 0 < c < {\hat u}_k(s,\theta,t) \leq 1, \qquad (s,\theta,t) \in K$$
if $k$ is chosen sufficiently large so that $K \subset \R \times [0,2\pi] \times (-\infty,|t_k|-1)$. Standard parabolic PDE arguments
imply that, passing to a subsequence,  ${\hat u}_k \to \tilde {\hat u}$ in $C^\infty$  on compact sets of $\R \times [0,2\pi] \times (-\infty,\infty)$. 
The function $\tilde {\hat u}$ is a smooth eternal solution of equation \eqref{eqn-hatu} on $\R \times [0,2\pi] \times (-\infty,\infty)$. 

\smallskip 
We will  next show that $\tilde  {\hat u} \equiv \gamma$, for some constant 
$\gamma$, which implies   that $\lim_{k \to \infty} Q(r_k,\theta_k,t_k)=0$, contradicting
our  assumption \eqref{eqn-QPk}.  

\begin{claim} If $R(s,\theta,t)$ is the scalar curvature in cylindrical coordinates, then we have 
\begin{equation}
\label{eqn-Rkkk}
\lim_{t_k \to - \infty} R(s_k,\theta_k,t_k) =0.
\end{equation}
\end{claim} 
\begin{proof}
Assume the claim is not true, that is, there exists a $\delta > 0$ and a subsequence $(s_k, \theta_k, t_k)$ so that $R(s_k,\theta_k,t_k) \ge \delta > 0$, for all $k$. Passing to a subsequence,  $\theta_k \to \theta_0$.
Since  $R_k:=- \Delta \,  {\log \hat  u}_k/{\hat u}_k$ satisfies 
$R_k(s,\theta,t)=R(s+s_k,\theta,t+t_k)$ and $R_k \to \tilde R:= - \Delta \log \tilde {\hat u} / \tilde {\hat u}$ uniformly on compact sets,  we conclude that 
$\tilde R (0,\theta_0,0):=\lim_{k \to \infty} R_k(0,\theta_k,0) \geq \delta.$
 
It follows that there  exists an $\epsilon > 0$ and $k_0$ so that 
\begin{equation}
\label{eq-lower-R11}
R(s+s_k,\theta,t_k) \ge \frac \delta 2, \qquad \mbox{for all}\,\, (s,\theta) \in I_{\epsilon} := [-\epsilon,\epsilon]\times [\theta_0 - \epsilon, \theta_0 + \epsilon].
\end{equation}
On the other hand, as we have proved earlier, we have $0 < c < \hat{u}_k(s,\theta,0) \le 1$, for all $(s,\theta)\in I_{\epsilon}$. Combining this with (\ref{eq-lower-R11}) yields 
$$\iint_{I_{\epsilon}} R_k (\cdot,t_k)\, \hat{u}_k (\cdot,0) \, ds\,d\theta \ge \hat{\delta}> 0\qquad k\ge k_0$$
or equivalently, 
\begin{equation}
\label{eqn-total-curv}
\iint_{I_{\epsilon}(s_k)} R(\cdot,t_k) \, \hat u (\cdot,t_k)  \, ds\,d\theta \ge \hat{\delta} > 0\qquad k\ge k_0
\end{equation}
where $I_{\epsilon}(s_k) :=[s_k-\epsilon,s_k+\epsilon]\times [\theta_0 - \epsilon, \theta_0 + \epsilon]$. 

\smallskip

Recall that by Lemma \ref{lem-cigar},  $\bar{u}_k(x,y,t) :=\rho_k^2 \,  \bar{u}(\rho_k x, \rho_k y, t+t_k)$ converges uniformly on compact subsets of $\mathbb{R}\times (-\infty,\infty)$ to a cigar solution. This implies that there exists a compact ball $B(0,\bar{r})$ (with $\bar r$ sufficiently large depending on $\eta$ and  $\eta$ chosen arbitrarily small) so that
$$\big |\iint_{B(0,\bar{r})} R_k \bar{u}_k (\cdot,0)  \, dx \, dy- 4\pi \big | < \eta$$ 
or equivalently, 
$$\big |\iint_{B(0,\rho_k \bar{r})} R(\cdot,   t_k)\, \bar{u}(\cdot,t_k)\, dx\, dy - 4\pi \big | < \eta.$$
We may also choose $\bar r$ so that $\bar r \geq \hat r$, where $\hat r$ is chosen as before so that 
 \eqref{eqn-bu32} holds. 
 
Set $r_k = e^{s_k}$ and $\bar{r}\rho_k = e^{\hat{s}_k}$. 
Recall that since we are in the case where $r_k / \rho_k \to +\infty$, we may assume that  $\hat s_k < s_k-1$. The last integral inequality in  cylindrical coordinates gives 
\begin{equation}
\label{eqn-oneside}
\big |\int_{-\infty}^{\hat {s}_k} \int_0^{2\pi} R (\cdot ,t_k)\,  \hat u(\cdot,t_k)\, d\theta \, ds - 4\pi \big | < \eta.
\end{equation}
Combining  (\ref{eqn-total-curv}) with  (\ref{eqn-oneside}) and  we choosing  $\eta << \hat{\delta}$
we obtain that  
\begin{equation}
\label{eqn-oneside2}
\int_{-\infty}^{s_k+\e}\int_0^{2\pi} R (\cdot,t_k)\hat u (\cdot,t_k)  \, d\theta\, ds > 4\pi + \frac \eta 2.
\end{equation}

Recall that $s_k \to - \infty$. Lemma \ref{lem-cigar} may be applied near the north pole $N$ of $S^2$  corresponding to
$\psi=\pi/2$, to also conclude that, after passing to a subsequence,  the rescaled solutions converge to a cigar.
In our chosen cylindrical coordinates this would imply,  that for the given sequence of times $t_k \to -\infty$, 
after passing to a subsequence,  there exists a sequence $\tilde s_k \to + \infty$  
for which
\begin{equation}
\label{eqn-otherside}
\big |\int_{\tilde  {s}_k}^{+\infty}  \int_0^{2\pi} R (\cdot ,t_k)\,  \hat u(\cdot,t_k)\,  d\theta \, ds  - 4\pi \big | <
\frac  \eta 4.
\end{equation}
Combining (\ref{eqn-oneside2}) with  (\ref{eqn-otherside}) we conclude  that the total curvature
$$\int_{-\infty}^{+\infty}\int_0^{2\pi} R (\cdot,t_k)\hat u (\cdot,t_k)  \, d\theta\, ds > 8 \pi$$
which is a contradiction  to the total curvature of our evolving compact surface  being equal 
always to $8\pi$. This concludes the proof of the claim.

\end{proof}

To finish the proof of the lemma, we will first show  that  $\tilde {\hat u}\equiv \gamma$, for a constant $\gamma >0$. To this end,
we will first prove  that the scalar curvature  $\tilde   R:= - \Delta \log \tilde  {\hat u} / \tilde  {\hat u}$ of the  metric $\tilde g := \hat {\hat u} \, (ds^2 +d\theta^2)$
is identically equal to zero. Clearly $\tilde  R \geq 0$. If we prove that $\tilde  R (0,\theta_0, 0) =0$, 
for some point $\theta_0 \in [0,2\pi]$, then $\tilde  R \equiv 0$ by the strong maximum
principle. But this readily follows from   \eqref{eqn-Rkkk} by choosing a   subsequence so that $\theta_k \to \theta_0$
and passing to the limit,  similarly as in the proof of the previous claim. 

To conclude that $\tilde {\hat u}$ is a constant, for a fixed $t$, set $w:= \log \tilde {\hat u}$ and observe that $w$ 
satisfies $\Delta_c w=0$ and  $w \leq 0$ on $\R \times [0,2\pi]$ (since $\hat u_k \leq 1$). 
We may view  $w$ as a harmonic function on $\R^2$ by extending it in the $\theta$ direction
so that it remains $2\pi$ periodic. The bound $w \leq 0$ then implies that $w$ must be a constant function, which shows that $\log \tilde {\hat u}(\cdot,t)=c(t)$, for all $t$. Since $R \equiv 0$,
we conclude that $c(t)$ is constant in $t$, hence  $\log \tilde {\hat u}\equiv c$.

\smallskip
We will now   conclude the proof of  Lemma \ref{lem-limit}. We   have just shown that ${\hat u}_k:=\hat u(s+s_k,\theta,t_k) \to \gamma$, for some constant $\gamma >0$, and the convergence  is 
in $C^\infty$ on compact subsets of $\R \times [0,2\pi]$.
Going back to the plane coordinates we conclude that $u_k:= r_k^2 \, u(r_k r,\theta,t_k) \to \gamma /r^2$
in $C^\infty$  on compact subsets of the punctured plane  $0 < r < \infty$. 
Notice that $\gamma/r^2$ is the cylindrical metric in plane coordinates. Since our quantity $Q$ is dilation invariant
and vanishes identically  on the cylinder,  this implies  that $Q(r_k,\theta, t_k) \to 0$ which contradicts \eqref{eqn-QPk}. 

\end{proof}

As an immediate consequence of Lemmas \ref{lem-mpQ} and \ref{lem-limit} we obtain

\begin{cor}\label{cor-ident} We have $Q(\cdot,t) \equiv 0$, for all $-\infty < t < 0$. Consequently,  the pressure function $\bv:=\bar v_N$
satisfies the identities
\begin{equation}\label{eqn-ide1} 
(a) \,\,\, {\bv}_{xxx}=3\, \bv_{xyy} \quad \mbox{and} \quad (b) \,\,\,  \bv_{yyy}=3\, \bv_{xxy} 
\end{equation}
The above identities  also imply the identities 
\begin{equation}\label{eqn-ide2} 
(a) \,\, \, \bv_{xxxx}= \bv_{yyyy} \quad \mbox{and} \quad  (b) \,\,\,  \bv_{xxxy}= \bv_{yyyx}=0. 
\end{equation}

\end{cor}

We will now show that  $\bv(\cdot,t)$ must be a fourth order polynomial of  a certain form.

\begin{lem}\label{lem-poly}  Let  $\bv(x,y)$ be a smooth function on $\R^2$ 
 satisfying \eqref{eqn-ide1}. Then, $\bv$ has the form
$$\bv(x,y)= a\,  ( (x-x_1)^2+(y-y_1)^2)^2 + q(x,y) $$
for some constants $a,x_1,x_2$ and  a quadratic polynomial $q(x,y)$. 
\end{lem}

\begin{proof} We will omit the details of calculations that can be checked in a straightforward manner by the reader.
We will also denote by $C,C_i$ various fixed constants. 
Identity \eqref{eqn-ide2}-(b) implies that $\bv_{xxx}=f_1(x)$ and $\bv_{yyy}=g_1(y)$, and by  \eqref{eqn-ide2}-(a)
we have $f_1(x)=C \, x+ C_1$, $g_1(y) = C\, y + C_2$, hence
$$\bv_{xx} = \frac C2 \, x^2 + C_1 \, x + g_2(y), \qquad \bv_{yy} = \frac C2 \, y^2 + C_2 \, y + f_2(x).$$
Combining the above identities with \eqref{eqn-ide1}, gives  
$$\bv_{xx} = \frac C2 x^2 + \frac C6 y^2 + C_1 x + \frac {C_2} 3 y + C_3, \quad
\bv_{yy} = \frac C2 y^2 + \frac C6 x^2 + C_2 y + \frac {C_1} 3 x + C_4.$$
Differentiating  these last  identities in $y,x$ respectively, gives
$$\bv_{xxy} = \frac C3 y + \frac{C_2}3, \qquad \bv_{xyy} = \frac C3 x + \frac{C_1}3$$
which after integration in $x,y$ respectively yield
$$\bv_{xy} = \frac C3 xy + \frac{C_2}3 x + g_3(y) = \frac C3 xy + \frac{C_1}3 y + f_3(x).$$
It follows that 
$$\bv_{xy} = \frac C3 xy + \frac{C_2}3 x + \frac{C_1}3 y + C_5.$$
If we set $q := \bv - V$, where 
$$V(x,y) = a\, ((x-x_1)^2+(y-y_1)^2)^2$$
with
$$a=\frac C{24}, \,\, x_1=- \frac {C_1}{24 a}, \, \,  y_1=- \frac {C_2}{24 a}$$
then a direct computation shows that $q$ satisfies
$$q_{xx}=C_3, \quad q_{yy}=C_4, \quad q_{xy}=C_5$$
from which the lemma readily follows. 
\end{proof}

We will next show that our solution $v$ has the particular form of the King-Rosenau solutions.

\begin{lem} Let $\bv(x,y,t)$ be an ancient solution of the equation
\begin{equation}\label{eqn-bvv} \bv_t = \bv \, \Delta \bv  - |\nabla \bv|^2\qquad \mbox{on} \,\, \R^2 \times (-\infty,0)
\end{equation}
of the form
$$ 
\bv(x,y,t)= a \, ( (x-x_1)^2+(y-y_1)^2)^2 + b\, (x-x_2)^2+ d \, (y-y_2)^2 + \rho \, xy  + c 
$$
where all $a,b,c,d,\rho$ and $x_i,y_i$ are functions of $t$. Assume in addition that 
\begin{equation}\label{eqn-ltv}
\lim_{t \to -\infty} \bv(x,y,t) = x^2+y^2
\end{equation}
uniformly on compact subsets of $\R^2$. Then, 
\begin{equation}\label{eqn-form1} 
\bv(x,y,t)= a(t)  \,  ( x^2+y^2)^2 + b(t) \,  (x^2+y^2)    + c(t) 
\end{equation}
for some functions of time $a(t), b(t),c(t)$ which are defined on $-\infty < t <0$.   

\end{lem}

\begin{proof} The lemma follows from a direct calculation where you plug  
a solution $\bv(x,y,t)$ of the given form   into the equation \eqref{eqn-bvv}
and compute the relation between all coefficients $a,b,c,d,\rho$ and
$x_i,y_i$. 

Indeed, by doing so we first find the following equations relating the coefficients
$a,b,c,d,\rho$: 
\begin{equation}\label{eqn-ode1} 
a' = 2a \, (b+d), \quad (b-d)' = -2 \, (b-d) \, (b+d), \quad \rho' = - 2 \rho \, (b+d).
\end{equation}
From \eqref{eqn-ltv} we have
$$\lim_{t \to -\infty} a(t)= \lim_{t \to -\infty} c(t)=\lim_{t \to -\infty} \rho(t)=0, \quad \lim_{t \to -\infty} b(t)= \lim_{t \to -\infty} d(t)=1$$
which,  in particular,  imply that $1 \leq b+d \leq 3$,  if $t < t_0$  with $t_0$ sufficiently close to $-\infty$.
Hence, the last two equations  in \eqref{eqn-ode1} readily  imply that $b\equiv d$ and $\rho \equiv 0$. 
Hence, $\bv$ is now of the simpler form
$$
\bv(x,y,t)= a \,  ( (x-x_1)^2+(y-y_1)^2)^2 + b\,  ( (x-x_2)^2+ (y-y_2)^2)  + c 
$$
where all $a,b,c$ and $x_i,y_i$ are all functions of $t$. Observe that since $\bv(x,y,t) >0$ on $\R^2 \times (-\infty,0)$ and $\lim_{t \to -\infty} b(t)=1$,
all coefficients $a, b, c$ are positive and  $3/4 \leq b(t) \leq 5/4$, for $t \leq  t_0 <0$. 
By \eqref{eqn-ode1}, we now have
$$
a' = 4\, a b \leq \, 5 \, a,  \qquad t \leq  t_0 <0
$$
which readily gives  the bound
\begin{equation}\label{eqn-phi13}
a(t) \geq  C_1 \, e^{5t}\end{equation}
for a constant $C_1 >0$. 
Now, plugging $\bv$ back to the equation, we find  by direct calculation, that 
\begin{equation}\label{eqn-x1y1} 
x_1' = - 4 b \, (x_1-x_2)\qquad y_1' = - 4 b \, (y_1-y_2)
\end{equation} 
and that 
$X(t):=x_1(t)-x_2(t)$ and $Y(t):=y_1(t)-y_2(t)$ both satisfy the same equation
$$X' = - \frac {4X}{b}   \, ( b^2 + 4ac + 4ab\,  (X^2+ Y^2)) $$
and the same for $Y$. It follows that $\phi(t):= X^2 + Y^2 >0$ satisfies the equation
\begin{equation}\label{eqn-phi10}
\phi ' = - \frac {8\phi}{b}   \, ( b^2 + 4ac + 4ab\,  \phi)
\end{equation}
where  $b^2 + 4ac + 4ab\,  \phi \geq b^2 > 0$ for $t < t_0$. Since $\lim_{t \to -\infty} b(t) =1$,
we have $3/4 \leq b(t) \leq 5/4$, for $t < t_0 <0$. It follows from \eqref{eqn-phi10} that
$$
\phi ' =   - 8 \, \phi  b \leq -6 \, \phi  b, \qquad  t \leq t_0 < 0
$$
which implies the bound 
\begin{equation}\label{eqn-phi11}
\phi(t) \geq C_2\, e^{6|t|}, \qquad  t \leq t_0 < 0
\end{equation}
for a constant $C_2 >0$, unless $\phi \equiv 0$. 

We will next show that $\phi \equiv 0$. Observe first that from \eqref{eqn-ltv} and the fact
that $\lim_{t \to -\infty} b(t)=1$, we have 
$$\lim_{t \to -\infty} a(t)\, (x_1^2(t)+ y_1^2(t))^2 = 0, \qquad  \lim_{t \to -\infty}  (x_2^2(t)+ y_2^2(t))=0
$$
which yields 
\begin{equation}\label{eqn-phi12}
\lim_{t \to -\infty} a(t) \, \phi^2(t) =0.
\end{equation}
On the other hand, it follows from \eqref{eqn-phi11} and \eqref{eqn-phi13} that 
$$ a(t) \, \phi^2(t) \geq C\, e^{5t + 12|t|}  = C\, e^{7|t|}$$
which contradicts  \eqref{eqn-phi12}. Hence, $\phi \equiv 0$.
Once we know that $\phi \equiv 0$,   \eqref{eqn-x1y1} and \eqref{eqn-ltv}  yield $x_1(t) = x_2(t) = 0$ and $y_1(t) = y_2(t) = 0$ for all $t$. 

We conclude from the   above discussion  that the solution $\tv$ is of the form \eqref{eqn-form1}.
\end{proof}

We will now conclude that our solution is one of the King-Rosenau solution in plane coordinates.
Such solutions  were first discovered  by King \cite{K1}. 

\begin{lem} Let $\bv(x,y,t)$ be an ancient solution of the equation
\eqref{eqn-bvv} of the form \eqref{eqn-form1}. Then, up to a dilation constant, which makes $a(t)= c(t)$ for all $t$, we have
\begin{equation}\label{eqn-cab}
a(t) = - \frac{\mu}2  \, \csch ( 4 \mu t) \quad \mbox{and} \quad  b(t) = - 
\mu   \, \coth (4 \mu t).
\end{equation}
\end{lem}

\begin{proof} If we plug a solution of the  form \eqref{eqn-form1} into the 
equation \eqref{eqn-bvv} we find  that the coefficients $a,b,c$ must satisfy the equations 
\begin{equation}\label{eqn-ode2} 
a' = 4 \, b\, a, \qquad c' = 4 \, b \, c, \qquad b'=16\, a\, c.
\end{equation}
Since  $a(t)>0$ and $c(t) >0$ the first two equations imply that 
$$(\log a(t))' = (\log c(t))'$$
which shows that 
$$c(t) = \lambda^2 \, a(t)$$
for a constant $\la >0$. By performing a dilation $\bv_\la(x,y,t) = \la^{-2} \, \bv (\la x,\la y,t)$ (which leaves $b(t)$
unchanged) we may assume that $\la =1$, i.e. $a \equiv c$. 
The functions $a, b$ satisfy the system
\begin{equation}\label{eqn-ode3} 
a' = 4 \, b\, a\qquad \mbox{and}  \qquad b'=16\, a^2.
\end{equation}
Solving this  system gives us \eqref{eqn-cab} for a given constant $\mu >0$ (if we assume that $\lim_{t \to -\infty} b(t)=1$,
then $\mu=1$). 

\end{proof}

We will now conclude the proof of Theorem \ref{prop-King}.

\begin{proof}[Proof of Theorem \ref{prop-King}] We observe  that if $\bv(r,t) = a(t) \, r^4 + b(t) \, r^2 + a(t)$
is the King-Rosenau solution in polar coordinates, then in cylindrical coordinates it takes the form 
$$\hat v (s,t)= 2 a(t) \cosh^2 s + b(t).$$ 
Recalling that $a(t)$ and $b(t)$ are given by \eqref{eqn-cab} and using  \eqref{eqn-uU}, we conclude,  by direct calculation, that 
$$v(\psi,t) =  - \mu  \coth (2\mu t)   + \mu \,   \tanh (2\mu t) \, \sin^2 \psi$$
finishing the proof of the  theorem. 
\end{proof}

\section{The Contracting Spheres}
\label{section-sphere}

Throughout this section we will assume that the backward lmit
\begin{equation}\label{eqn-zero}
v_\infty := \lim_{t \to - \infty} v(\cdot,t) \equiv 0.
\end{equation}
Our goal is to show that in this case the ancient solution $v$ must be a  family of contracting spheres, 
as stated in the following theorem.

\begin{thm}
\label{prop-sphere}
If  the backward limit $v_\infty\equiv 0$, then $$v(\cdot,t) = \frac{1}{(-2t)}$$
that is, our ancient solution is a family of contracting spheres.
\end{thm}

To prove the theorem we will use an isoperimetric estimate for the Ricci
flow which was proven by R. Hamilton in \cite{Ha}.  Let $M$ be any compact surface. Any simple closed curve $\gamma$ 
on  $M$  of  length $L(\gamma)$ 
divides the  compact  surface $M$  into two regions   with areas $A_1(\gamma)$ 
and  $A_2(\gamma)$. We define the isoperimetric ratio 
as in \cite{HT}, namely
\begin{equation}
\label{eq-isop}
I = \frac 1{4\pi} \, \inf_{\gamma} L^2(\gamma) \, \left ( \frac 1{A_1(\gamma)} + \frac 1{A_2(\gamma)} \right ).
\end{equation}
It is well known that  $I \leq 1$ always,  and that $I\equiv 1$ if and only if the surface $M$ is a sphere. 

We will briefly outline the proof of  Theorem \ref{prop-sphere} whose steps will be proven  in detail  afterwards.
We consider  our evolving surfaces at each time $t < 0$, and define the isoperimetric ratio $I(t)$ as above. Our goal is to show that  our assumption \eqref{eqn-zero} implies that   $I(t) \equiv 1$, which forces  $(M,g(t))$  to be a family of contracting  spheres. We will argue by contradiction and assume that  $I(t_0) < 1$, for some $t_0 <0$. 
In that case  we will  show that there exists a sequence $t_k \to - \infty$ and   closed curves $\beta_k$ on $S^2$ so that  simultaneously      we have 
\begin{equation}
\label{eq-length}
L_{S^2}(\beta_k) \ge \delta > 0 \quad  \mbox{and} \quad L_{g(t_k)}(\beta_k) \le C\qquad 
\forall k
\end{equation}
where  $L_{S^2}$ and $ L_{g(t_k)} $ denote   the length of a curve in the round metric on $S^2$ and  in the metric $g(t_k)$, respectively. This clearly contradicts  the fact that  $u(\cdot,t_k)  \to \infty$, uniformly in $S^2$ ( implied  by \eqref{eqn-zero}) and finishes the proof. 

We will now outline how we will find the curves $\beta_k$. For each $t < t_0$, let  $\gamma_t$ be a  curve for which the isoperimetric ratio $I(t)$ is achieved. 
\begin{enumerate}[i.]
\item
If $I(t_0) < 1$, for some $t_0 <0$, we will show that   $I(t) \le \frac{C}{|t|}$,  for $t < t_0$. We 
will use that to show $L_{g(t)}(\gamma_t) \le C$, for all $t < t_0$.  
\item
For any sequence $t_k \to -\infty$ and $p_k \in \gamma_{t_k}$,  we will show that there exists a subsequence  
such that  $(M,g(t_k),p_k)$   converges to $(M_{\infty},g_{\infty},p_{\infty})$, where $M_{\infty} = S^1\times \mathbb{R}$ and $\gamma_{\infty} := \lim_{k\to\infty}\gamma_{t_k}$ is a closed geodesic on $M_{\infty}$, one of the cross circles of $S^1\times \mathbb{R}$.
\item
Let $t_k$ be as above. If $A_1(t_k), A_2(t_k)$ are the areas of the two regions  into which $\gamma_{t_k}$ divides $S^2$, we show that both of them are comparable to $|t_k| = -t_k$.
\item
We show that the maximal distances from $\gamma_{t_k}$ to the points of  the two 
regions of areas $A_1(t_k), A_2(t_k)$ respectively  are both of length comparable to $|t_k|$.
\item
The curves $\gamma_{t_k}$ do not necessarily satisfy (\ref{eq-length}). However, we use
 them and (ii) to define  a foliation $\{\beta_{w}^k\}$ of our surfaces $(M,g(t_k))$ and we  choose the curve $\beta_k$ from this foliation that splits $S^2$ into two parts of equal areas with respect to the round metric. We prove that  this is the curve that satisfies (\ref{eq-length}) by  
 using  that $I_{S^2} = 1$, the Bishop Gromov volume comparison principle,   (iii) and  (iv).
\end{enumerate}

\begin{lem}
\label{lem-isop}
If $I(t_0) < 1$, for some $t_0 <0$,  then there exist positive constants $C_1, C_2$ so  that 
$$I(t) \le \frac{C_1}{|t|+C_2}\qquad \mbox{for all} \,\,\, t < t_0.$$
Moreover, if $\gamma_t$ is the curve at  which the infimum in (\ref{eq-isop}) is attained
then,  
$$L(t) := L(\gamma_t) \le C\qquad \mbox{for all}\,\,\, t < t_0.$$
\end{lem}

\begin{proof}
Let $t < t_0$, with $t_0$ as in the statement of the lemma. It has been shown in  \cite{Ha} that  
$$I'(t) \geq \frac{4\pi\, (A_1^2+A_2^2)}{A_1A_2(A_1+A_2)} \, I\, (1-I^2).$$
Since $A_1+A_2= 8 \pi |t|$ and $A_1^2+A^2_2 \geq 2A_1A_2$, we conclude the
differential inequality
$$I'(t) \geq  \frac 1{|t|} \, I\, (1-I^2).$$
Since $I(t_0) < 1$, the above inequality implies the bound
$$I(t) \leq \frac {C_1}{|t|  + C_2}\qquad \mbox{for all} \,\, t < t_0$$
for  uniform in time constants $C_1$ and $C_2$.  Using that  
$ \frac 1{A_1} + \frac 1{A_2} \geq \frac  1{4\pi |t|}$, 
we will conclude that  that  the   length $L(t)$ of a  curve $\gamma_t$ at which the infimum in \eqref{eq-isop} is attained satisfies 
$L(t) \leq C$, for all $t < t_0$. 

\end{proof}

We also have the following estimate from below on the length   $L(t)$ of the curve at which the infimum in \eqref{eq-isop} is attained. 

\begin{lem}
\label{lem-lower-l}
There is a uniform constant $c > 0$, independent of time so that
$$L(t) \ge c\qquad  \mbox{for all} \,\,\, t\le t_0 < 0.$$
\end{lem}

\begin{proof}
Recall that for $t_0 < 0$, the scalar curvature $R$ satisfies  $0 < R(\cdot,t) \le C$,  for all $t \le t_0$.
The Klingenberg injectivity radius estimate for even dimensional manifolds implies the bound
\begin{equation}
\label{eq-injrad}
\injrad(g(t)) \ge \frac{c}{\sqrt{R_{\max}}}\ge \delta > 0\qquad  \mbox{for all} \,\,\, t \le t_0 < 0
\end{equation}
for a uniform in time constant $c >0$. 
We will prove the Lemma by contradiction. Assume that there is a sequence $t_i\to -\infty$,  so that $L_i := L(t_i) \to 0$,  as $i\to\infty$,  and denote by $\gamma_{t_i}$ a  curve
at which the isoperimetric  ratio is attained, i.e. $L(t_i)=L(\gamma_{t_i})$. 

Define a new sequence of  re-scaled Ricci flows,  $g_i(t) := L_i^{-2}\,  g(t_i+ L_i^2\, t)$ and take a sequence of points $p_i\in \gamma_{t_i}$.  The bound   (\ref{eq-injrad})  implies a lower bound on the injectivity radius at $p_i$ with respect to metric $g_i$, namely 
\begin{equation}
\label{eq-in1}
\injrad_{g_i}(p_i) = \frac{\injrad_{g(t_i)}(p_i)}{L_i^2} \ge \frac{\delta}{L_i^2} \to \infty\qquad  \mbox{as} \,\,\, i\to\infty.
\end{equation}
Also, since $R_i(\cdot,t) = L_i^2 \, R(\cdot, t_i+L_i^2\, t) \leq C\, L_i^2$ and $L_i \to 0$, we get  
\begin{equation}
\label{eq-R2}
\max R_i(\cdot,t)   \to 0\qquad  \mbox{as} \,\,\, i\to\infty.
\end{equation}
Hamilton's  compactness theorem (c.f. in \cite{H2}) implies, passing to a subsequence,  the 
  pointed smooth convergence of $(M,g_i(0),p_i))$ to a complete manifold $(M_{\infty},g_{\infty},p_{\infty})$,  which is,  due to (\ref{eq-in1}) and (\ref{eq-R2}),  a standard plane. 
Moreover, 
$$I(t_i) = \frac{1}{4\pi}L_i^2\, \left (\frac{1}{A_1(t_i)} + \frac{1}{A_2(t_i)} \right ) = \frac{1}{4\pi}\left(\frac{1}{A_1(g_i(0))} + \frac{1}{A_2(g_i(0))}\right)$$
where $A_1(g_i(0))$ and $A_2(g_i(0))$ are the areas inside  and outside the curve $\gamma_{t_i}$, respectively, both computed with respect to metric $g_i(0)$. Since $g_i(0)$ converges to the euclidean metric and $\gamma_{t_i}$ converges to a curve of length $1$, it follows that 
$\lim_{i\to\infty}A_1(g_i(0)) = \alpha > 0$ and  $\lim_{i\to\infty}A_2(g_i(0)) = \infty$,   which implies that 
$$\lim_{i\to\infty} I(t_i) \ge \delta > 0$$
and  obviously contradicts Lemma \ref{lem-isop}.
\end{proof}

We recall that  each time $t$,  a curve $\gamma_t$   at which the
isoperimetric ratio  is achieved  splits the surface into two regions of areas  $A_1(t)$ and $A_2(t).$ 
Lemma \ref{lem-lower-l} yields to  the following conclusion. 

\begin{cor}
\label{claim-area-t}
There are  uniform constants $c > 0$ and $C > 0$ so that 
$$c\, |t| \leq A_1(t) \le C\, |t| \quad \mbox{and}  \quad c\, |t| \leq A_2(t) \le C\, |t|$$
for all $t < t_0 <0$. 
\end{cor} 

\begin{proof}
It is well known that the total area  of our evolving surface is $A(t)= 8\pi\, |t|$.
Hence,  $A_1(t) \le 8\pi\, |t|$ and $A_2(t) \le 8\pi\, |t|$. On the other hand, by
Lemmas  \ref{lem-isop} and \ref{lem-lower-l}, we have 
$$
\frac{c}{A_j(t)} \leq   \frac{L^2(t)}{A_j(t)} \leq  I(t) \leq \frac C{|t|}\qquad j=1,2
$$
for all $t < t_0$, which shows  that 
$A_j(t) \ge c\, |t|$, $j=1,2$, 
for a uniform constant $c > 0$, therefore proving the corollary.
\end{proof}

We will fix in the sequel a sequence  $t_k \to -\infty$. Let $\gamma_{t_k}$ be,  as
before,  a curve at which the isoperimetric ratio is achieved. 
From now on we will refer to  $\gamma_{t_k}$ as an isoperimetric
curve at time $t_k$. To simplify the notation,
we will set $A_{1k} :=
A_1(t_k)$, $A_{2k} := A_2(t_k)$ and $L_k = L(t_k)$. It follows from  Corollary \ref{claim-area-t} that
\begin{equation}
\label{eq-inf-area} 
\lim_{k\to\infty} A_{1k} = +\infty \qquad  \mbox{and} \qquad   \lim_{k\to\infty} A_{2k} = +\infty.
\end{equation}

Pick a sequence of points $p_k \in \gamma_{t_k}$ and
look at the pointed sequence of solutions $(M,g(t_k+t),p_k)$.  Since
the curvature is uniformly bounded and since the injectivity radius at
$p_k$ is uniformly bounded from below, by Hamilton's compactness
theorem we can find a subsequence of pointed solutions that converge,
in the Cheeger-Gromov sense,  to a complete  smooth solution
$(M_{\infty},g_{\infty},p_{\infty})$. This means that for every compact set $K\subset M_{\infty}$ there are compact sets $K_k\subset M$ and diffeomorphisms $\phi_k:K\to K_k$ so that $\phi_k^*g(t_k)$ converges to $g_{\infty}$. From Lemma
\ref{lem-isop},   $L(t_k) \le C$, for all $k$, and therefore our
curves $\gamma_{t_k}$ converge to a curve $\gamma_{\infty}$ (this
convergence is induced by the  manifold convergence) which by
(\ref{eq-inf-area}) has the property that it splits $M_{\infty}$ into
two parts (call them $M_{1\infty}$ and $M_{2\infty}$), each of which
has  infinite area. It follows that we can choose points $x_j \in
M_{1\infty}$ and $y_j\in M_{2\infty}$ so that
$\dist_{g_{\infty}}(x_j,p_{\infty}) =
\dist_{g_{\infty}}(p_{\infty},y_j) = \rho_j$, where $\rho_j$ is an
arbitrary sequence so that $\rho_j \to \infty$. Since $(M_{\infty},
g_{\infty})$ is complete, there exists a minimal geodesic $\beta_j$
from $x_j$ to $y_j$. This geodesic $\beta_j$ intersects
$\gamma_{\infty}$ at some point $q_j$. Since $q_j\in \gamma_{\infty}$
and $\gamma_{\infty}$ is a closed curve of finite length, the set $\{q_j\}$
is compact and therefore there is a subsequence so that $q_j \to
q_\infty \in \gamma_{\infty}$. This implies that  there is a subsequence of
geodesics $\{\beta_j\}$ so that, as $j\to \infty$,  it  converges
to a minimal geodesic $\beta_\infty: (-\infty,\infty)\to M_{\infty}$ (minimal
geodesic means a globally distance minimizing geodesic).  It follows that 
our limiting manifold $M_{\infty}$ contains a straight line. Since the
curvature of $M_{\infty}$ is zero, by the splitting theorem our
manifold splits off a line and therefore is diffeomorphic to the 
cylinder $S^1\times \mathbb{R}$.

We next observe that the limiting curve $\gamma_{\infty}$ is a geodesic,
as shown  in the following lemma.

\begin{lem}
\label{lem-geodesic}
The geodesic curvature $\kappa$ of the curve $\gamma_{\infty}$ is zero.
\end{lem}

\begin{proof}
As in \cite{Ha}, at each time $t < t_0 <0$ we start with the isoperimetric curve $\gamma_t$  
and we construct the one-parameter family of parallel curves $\gamma_t^r$ at distance $r$  from $\gamma_t$ on either side. We take $r > 0$ when the curve moves from the region of area $A_1(t)$ to the region of area $A_2(t)$,
 and $r < 0$ when it moves the other way. We then regard $L, A_1, A_2$ and 
$$I = I(\gamma_t^r) = \frac{1}{4\pi}L^2(\gamma_t^r)\, \left (\frac{1}{A_1(\gamma_t^r)} + \frac{1}{A_2(\gamma_t^r)} \right )$$ 
as functions of $r$ and $t$. By the computation in \cite{Ha} we have
$$\frac{\partial  A_1}{ \partial r} = L\qquad  \frac{\partial  A_2}{\partial r} = -L\qquad \frac{d L}{dr} = \int \kappa\, ds = \kappa\, L$$
where $\kappa$ is the geodesic curvature of the curve $\gamma_t^r$. By a  standard variational argument $\kappa$ is constant on $\gamma_t$. If $A := A_1 + A_2$ is the total surface area, we have
$$\log I = 2\log L + \log A - \log A_1 - \log A_2 - \log(4\pi).$$
Since $\frac{\partial I}{\partial r}|_{r=0} = 0$, we conclude that 
\begin{eqnarray*}
0=  \frac{2}{L}\, \frac{\partial L}{\partial r} + \frac{1}{A}\, \frac{\partial A}{\partial r} - \frac{1}{A_1}\, \frac{\partial A_1}{\partial r} - \frac{1}{A_2}\, \frac{\partial A_2}{\partial r} =
  \frac{2}{L}\,  \kappa L - \frac{1}{A_1}\,  L + \frac{1}{A_2}\,L
\end{eqnarray*}
which leads to
$$\kappa = \frac{L}{2}\, (\frac{1}{A_1} - \frac{1}{A_2}).$$
By Lemmas \ref{lem-isop} and  \ref{lem-lower-l} and (\ref{eq-inf-area}) we conclude that 
$$\kappa_{\infty} := \lim_{t\to -\infty} \kappa = \lim_{t\to -\infty} \frac{L}{2}\, (\frac{1}{A_1} - \frac{1}{A_2}) = 0,$$
which means the geodesic curvature $\kappa_{\infty}$ of the limiting curve $\gamma_{\infty}$
is zero.
\end{proof}

We have just shown that our limiting manifold is a cylinder $M_{\infty} = S^1\times \mathbb{R}$  and $\gamma_{\infty}$ is a closed geodesic on $M_{\infty}$. Hence,  $\gamma_{\infty}$ is one of the cross circles of $M_{\infty}$.  

We have the following picture assuming  that  the radius of $\gamma_{\infty}$ is $1$.

\begin{figure}[htbp]
\begin{center}
 \input{slika.pstex_t}
 \label{fig:slika}
\end{center}
\end{figure}

Assume that we have a  foliation of our limiting  cylinder $M_\infty$ by circles $\beta_w$, where $|w|$ is the  distance from $\beta_w$ to $\gamma_{\infty}$, taking $w > 0$ if $\beta_w$ lies  on the upper side of the cylinder and $w < 0$ if $\beta_w$ lies on its lower side.  Denote by $\beta_w^k$  the curve on $M$  such that $\phi_k^*\beta_w^k = \beta_w$.

One of the properties of the  cylinder is that for every $\delta > 0$ there is a $w_0 > 0$ so that for every $|w| \ge w_0$ we have
$$|\sup_{x\in \beta_w, y\in \gamma_{\infty}}\dist(x,y) - \inf_{x\in \beta_w,y\in \gamma_{\infty}}\dist(x,y)| \le \sqrt{w^2+\pi^2} - |w| \le \frac{C}{|w|} \le \frac{C}{w_0} < \frac \delta 2$$
where the distance is computed in the cylindrical metric on $M_{\infty}$.

Since for every sequence $t_k\to -\infty$, there exists a   subsequence for which we have 
uniform convergence of our metrics $\{g(t_k)\}$ on  bounded sets around the points $p_k \in \gamma_{t_k}$, the previous observation implies  the following claim which will be used frequently from now on. 

\begin{claim}
\label{claim-compare}
For every sequence $t_k\to -\infty$ and every $\delta > 0$ there exists  $k_0$ and $w$ so that for $k \ge k_0$,
$$|\sup_{x\in \beta_k^w, y\in \gamma_{t_k}}\dist_{g(t_k)}(x,y) - \inf_{x\in \beta_k^w, y\in \gamma_{t_k}}\dist_{g(t_k)}(x,y)| < \delta.$$
\end{claim}


\medskip

The variant of the Bishop-Gromov volume comparison principle (since  $R \ge 0$) implies the following area comparison of the annuli, for each $t < 0$, 
\begin{equation}
\label{eq-comparison}
\frac{\ar(b_1 \le s \le b_2)}{\ar(a_1 \le s \le a_2)} \le \frac{b_2^2-b_1^2}{a_2^2-a_1^2}
\end{equation}
where $a_1 \le a_2 \le b_1 \le b_2$ and $s$ is the  distance from a fixed point on $(M,g(t))$, computed with respect to the metric $g(t)$. We are going to use this fact in the  lemma that follows. 


For each $k$, $\gamma_{t_k}$ splits our manifold in two parts, call them $M_{1k}$ and $M_{2k}$ with  areas $A_{1k}$ and $A_{2k}$ respectively. Choose points $x_k \in M_{1k}$ and $y_k\in M_{2k}$ so that
$$\dist_{g(t_k)}(x_k,\gamma_{t_k}) = \max_{z\in M_{1k}}\dist_{g(t_k)}(\gamma_{t_k},z) =: \rho_k$$
and
$$\dist_{g(t_k)}(y_k,\gamma_{t_k}) = \max_{z\in M_{2k}}\dist_{g(t_k)}(\gamma_{t_k},z) =: \sigma_k.$$
By the  definition of $\sigma_k$ and $\rho_k$ and  from the convergence of $(M,g(t_k),p_k)$ to an infinite cylinder,  we have
$$\lim_{k\to\infty}\sigma_k = +\infty  \qquad \mbox{and} \qquad  \lim_{k\to\infty}\rho_k = +\infty.$$
\begin{lem}
\label{lem-ar-below}
There are uniform constants  $k_0>0$ and  $c > 0$ so that the
$$ \ar \, (B_{\rho_k}(x_k)) \ge c\, \rho_k \quad \mbox{and} \quad \ar \,  (B_{\sigma_k}(y_k)) \ge c\, \sigma_k
,\qquad \mbox{for all} \,\,\, k \geq k_0$$ where both the distance and the area  are  computed with
respect to the metric $g(t_k)$.
\end{lem}

\begin{proof}
We take  $a_1 = 0$, $a_2 = b_1 = \sigma_k \geq 1$ and $b_2 = \sigma_k + 1$ in \eqref{eq-comparison}. Then, if $s$ is the distance from $y_k$ computed with respect to $g(t_k)$, we have 
$$\frac{\ar \, (\sigma_k \le s \le \sigma_k + 1)}{\ar \, (0 \le s \le \sigma_k)} \le \frac{(\sigma_k+1)^2 - \sigma_k^2}{\sigma_k^2} \le \frac{3}{\sigma_k}.$$
Hence, 
\begin{equation}
\label{eq-from-below}
\ar\, (B_{\sigma_k}(y_k)) \ge \frac{\sigma_k}{3}\,\, \ar \,  (\sigma_k \le s \le \sigma_k + 1).
\end{equation}
Having (\ref{eq-from-below}), the proof of Lemma \ref{lem-ar-below} is finished once we show the following estimate: there are uniform constants $c > 0$ and $k_1$ so that for $k \ge k_1$,
\begin{equation}
\label{eq-below1}
\ar \,  (\sigma_k \le s \le \sigma_k + 1) \ge c.
\end{equation}
To prove the estimate, denote by $U_k := \{z\,\, |\,\, \sigma_k \le s \le \sigma_k +1\}$. We consider the set 
$$V_k := \{z\,\, |\,\, \dist_{t_k}(z,\gamma_{t_k}) \le \frac{1}{2}, \,\,\, \overline{y_kz}\cap\gamma_{t_k} \neq \emptyset\}$$
where $\overline{y_kz}$ denotes a geodesic connecting the points $y_k$ and $z$.
It is enough to show  that  $V_k \subset U_k$, for $k$ sufficiently large, and that 
$\ar \, (V_k) \geq c >0$. 
To prove that  $V_k \subset U_k$, take $z\in V_k$ and let
$w_k\in \gamma_{t_k}$ be such that $\dist_{t_k}(z,w_k) = \dist_{t_k}(z,\gamma_{t_k}) \le \frac{1}{2}$.
If $q_k := \gamma_{t_k}\cap \overline{zy_k}$,  then
$$\sigma_k = \dist_{t_k}(y_k,\gamma_{t_k}) \le \dist_{t_k}(y_k,q_k) \le \dist_{t_k}(z,y_k)$$
which implies that $\sigma_k \le \dist_{t_k}(z,y_k)$. On the other hand, by Claim \ref{claim-compare} 
we have $\dist_{t_k}(w_k,y_k) \le \sigma_k + \frac{1}{2}$, 
for $k$ sufficiently large. Hence
$$\dist_{t_k}(z,y_k) \le \dist_{t_k}(y_k,w_k) + \dist_{t_k}(w_k,z) \le \sigma_k + \frac{1}{2} + \frac{1}{2}
< \sigma_k + 1$$
for $k$ sufficiently large. This proves that $V_k \subset U_k$ and hence 
$$\ar\,(U_k) \ge \ar\, (V_k).$$
To estimate   $ \ar\, (V_k)$ from below,  we recall that for $p_k \in \gamma_{t_k}$, we have
 pointed convergence  of $(M,g(t_k),p_k)$ to a cylinder which is uniform on compact sets
around $p_k$. To use this we need to show  there is a constant $C >0$, for which   
$$V_k \subset B_{t_k}(p_k,C)\qquad \mbox{for all} \,\,\, k\ge k_0.$$
Let $z\in V_k$ and let $q_k \in \overline{y_kz}\cap\gamma_{t_k}$. Then by Claim \ref{claim-compare}
for $k$ sufficiently large, we have
\begin{equation}
\label{eq-both-sides}
\sigma_k - 1 \le \dist_{g(t_k)}(y_k,q_k) \le \sigma_k + 1.
\end{equation}
We also have 
$$\dist_{g(t_k)}(p_k,z) \le \dist_{g(t_k)}(p_k,q_k) + \dist_{g(t_k)}(q_k,z).$$
Since, $z \in V_k \subset U_k$ and \eqref{eq-both-sides} holds, we get 
$$\dist_{g(t_k)}(q_k,z) \le \dist_{g(t_k)}(y_k,z) - \dist_{g(t_k)}(y_k,q_k) \le \sigma_k + 1 - \sigma_k + 1 = 2$$
which combined with 
$\dist_{g(t_k)}(p_k,q_k)\le L(\gamma_{t_k}) \le C$ gives us 
 the bound 
$$\dist_{g(t_k)}(p_k,z) \le C.$$

This guarantees that, as $k\to \infty$, $V_k$ converges to a part of
the cylinder $S^1\times \mathbb{R}$, while $(M,g(t_k),p_k) \to
(S^1\times\mathbb{R},g_\infty,p_{\infty})$ and $g_\infty$ is the cylindrical
metric. Recall that $\gamma_{t_k}\to \gamma_{\infty}$ and
$\gamma_{\infty}$ is one of the cross circles on $S^1\times \R$.  It
 follows that $V_k$ converges as $k\to\infty$ to the upper or lower part of the set  
$\{z\in S^1\times\R \,\, |\,\, \dist_{g_\infty}(z,\gamma_{\infty}) \le
\frac{1}{2}\}$ with respect to $\gamma_\infty$.  This implies that
\begin{equation}
\label{eq-around-p}
c \le \ar\, (\{\sigma_k \le s \le \sigma_k + 1\}) \le C\qquad  \mbox{for} \,\,\, k \ge k_0,
\end{equation}
for some uniform constants $c, C > 0$ finishing the proof of (\ref{eq-below1}) and therefore Lemma \ref{lem-ar-below}.
\end{proof}

Let us  denote briefly by $A_{\sigma_k} := \ar\, (B_{\sigma_k}(y_k))$ and $A_{\rho_k} := \ar\, (B_{\rho_k}(x_k))$.

\begin{lem}
\label{lem-ar-as}
There exist a number  $k_0$ and constants $c_1>0, c_2>0$,  so that
$$c_1\, |t_k| \le A_{\rho_k} \le c_2\,  |t_k|  \quad  \mbox{and} \quad c_1\, |t_k| \le A_{\sigma_k} \le c_2\,  |t_k|\qquad \mbox{for all} \,\,\,  k\ge k_0.$$
\end{lem}

\begin{proof}
Notice  that 
\begin{equation}
\label{eq-ar-id}
A_{\rho_k} + A_{\sigma_k} \le 2\, A(t_k) = 16 \pi\,  |t_k|
\end{equation}
since  $A(t_k) = 8\pi|t_k|$ is the total surface area. Hence, 
\begin{equation}
\label{eq-one-side-ar}
A_{\rho_k} \le C\, |t_k| \qquad \mbox{and} \qquad  A_{\sigma_k} \le C\, |t_k|.
\end{equation}
To establish  the bounds from below,   we will use Lemma
\ref{lem-ar-below} and show that  there is a uniform constant $c$ so that
\begin{equation}
\label{eq-comp-rstau}
\sigma_k \ge c\, |t_k| \qquad \mbox{and} \qquad  \rho_k \ge c\, |t_k|\qquad  \mbox{for all} \,\,\, k\ge k_0.
\end{equation}

We will first show there are uniform constants $c>0$ and $C < \infty$,  so that
\begin{equation}
\label{claim-comp-rs}
c\, \rho_k \le \sigma_k \le C \, \rho_k.
\end{equation}
Recall that $\sigma_k = \dist_{g(t_k)}(y_k,\gamma_{t_k})$. By our choice
of points $x_k, y_k$ and the figure we have that the
$\diam(M, g(t_k)) \le \sigma_k + \rho_k + 1$ for $k \ge k_0$, sufficiently
large. We also have that the subset of $M$ that corresponds to area $A_2(t_k)$ 
contains a ball $B_{\sigma_k}(y_k)$. By Corollary  \ref{claim-area-t} 
and the comparison inequality (\ref{eq-comparison}), we have 

\begin{eqnarray*}
c &\le& \frac{A_1(t_k)}{A_2(t_k)} \le \frac{\ar(B_{\rho_k+\sigma_k+1}(y_k)
\backslash B_{\sigma_k}(y_k))}{\ar(B_{\sigma_k}(y_k))}\\
&=& \frac{\ar(\sigma_k \le s \le \rho_k + \sigma_k+1)}{\ar(0 \le s \le \sigma_k)} \le 
\frac{(\rho_k+\sigma_k+1)^2 - \sigma_k^2}{\sigma_k^2}.
\end{eqnarray*}
Using the previous inequality we obtain the bound
$$c\, \sigma_k^2 - 2  \,  \rho_k - 2\sigma_k  -2\rho_k\, \sigma_k - 1 \le \rho_k^2.$$
We claim there is a uniform constant $c$ so that $\sigma_k \le c\rho_k$. If not, then  $\rho_k << \sigma_k$ for $k >>1$, and from the
inequality above we get
$$\frac{c}{2}\, \sigma_k^2 \le \rho_k^2\qquad \mbox{for}\,\,\,  k >>1.$$
In any case there are $k_1$ and $C_1 > 0$ so that
\begin{equation}
\label{eq-rs1}
\rho_k \le C_1\sigma_k\qquad \mbox{for} \,\,\, k\ge k_1.
\end{equation}
By a similar
analysis as above there are $k_2 \ge k_1$ and $C_2 > 0$ such  that
\begin{equation}
\label{eq-rs2}
\sigma_k \le C_2\rho_k\qquad \mbox{for} \,\,\, k\ge k_1.
\end{equation}

We will now conclude the proof of Lemma \ref{lem-ar-as}. 
By Lemma \ref{lem-ar-below} and (\ref{eq-ar-id}) it follows that
$$\rho_k + \sigma_k \le C|t_k|\qquad \mbox{for} \,\,\, k >>1.$$
By (\ref{eq-rs1}) and (\ref{eq-rs2}) it follows that 
$$\rho_k \le C\, |t_k|  \quad  \mbox{and} \quad  \sigma_k \le C\, |t_k|\qquad \mbox{for} \,\,\, k >>1.$$
Moreover, by (\ref{eq-comparison}), we have 
\begin{eqnarray}
\label{eq-lower-rho}
\frac{A_1(t_k)}{\ar\, (\sigma_k-1 \le s \le \sigma_k)} &\le&
\frac{\ar\, (B_{\rho_k+\sigma_k+1}(y_k)\backslash B_{\sigma_k}(y_k))}{\ar\, (\sigma_k-1 \le s \le \sigma_k)} \nonumber \\
&\le& \frac{(\rho_k+\sigma_k+1)^2 - \sigma_k^2}{\sigma_k^2 - (\sigma_k-1)^2} \le
\frac{(\rho_k+\sigma_k+1)^2-\sigma_k^2}{2\sigma_k-1} \nonumber \\
&\le& \frac{(\rho_k+\sigma_k+1)^2 - \sigma_k^2}{2\sigma_k - 1} \nonumber \\
&\le& C \, \rho_k
\end{eqnarray}
where we have used (\ref{eq-rs1}) and (\ref{eq-rs2}).  The same
analysis that yielded to (\ref{eq-around-p}) can be applied again to
conclude that
$$\ar(\sigma_k-1\le s \le \sigma_k) \le C.$$
This together with Corollary \ref{claim-area-t} and (\ref{eq-lower-rho}) imply
$$\rho_k \ge c\, |t_k|\qquad \mbox{for} \,\,\, k\ge k_0.$$
Claim \ref{claim-comp-rs} implies the same conclusion about $\sigma_k$.
This is sufficient to  conclude  the proof of Lemma \ref{lem-ar-as},  as we have explained at the beginning of it.
\end{proof}

We will now finish the proof of Theorem \ref{prop-sphere}.

\begin{proof}[Proof of Theorem \ref{prop-sphere}]
If the isoperimetric constant $I(t) \equiv 1$, it follows by a well
known result  that our solution is a family of contracting
spheres. Hence we will assume that $I(t_0) < 1$, for some $t_0 <0$, which implies all the
results in this section are applicable. We will show that this
contradicts the fact that $\lim_{t\to -\infty} v(\cdot,t) = 0$,
uniformly on $S^2$.

As explained at the beginning of this section, it suffices to find
positive constants $\delta, C$ and curves $\beta_k$, so that
\begin{equation}\label{eqn-claim-111}
L_{S^2}(\beta_k) \ge \delta > 0 \qquad \mbox{and} \qquad L_k(\beta_k) \le C< \infty
\end{equation}
where $L_{S^2}$ denotes the length of a curve computed in the round
spherical metric and $L_k$ denotes the length of a curve computed in
the metric $g(t_k)$. If we manage to find those curves $\beta_k$ that
would imply 
\begin{eqnarray*}
C \geq  L_k(\beta_k) = \int_{{\beta}_k}\sqrt{u(t_k)}\, d_{S^2} 
\geq  M\,  L_{S^2}(\beta_k) \ge M\, \delta, \quad \mbox{for} \,\,\, k \geq k_0
\end{eqnarray*}
where $d_{S^2}$ is the length element with respect to the standard round spherical metric, $M > 0$ is an arbitrary big constant and $k_0$ is sufficiently 
large so that $\sqrt{u(t_k)} \ge M$,  for $k\ge k_0$, uniformly on $S^2$
(which is justified by the fact $v(\cdot,t)$ converges uniformly to zero on
$S^2$, in $C^{1,\alpha}$ norm).  
The last  estimate is impossible, when $M$ is taken larger  than ${C}/{\delta}$,
hence finishing the proof of our theorem.

We will now prove \eqref{eqn-claim-111}. Our isoperimetric curves
$\gamma_{t_k}$ have the property that  $L_k(\gamma_{t_k}) \le C$ for all
$k$, but we do not know whether $L_{S^2}(\gamma_{t_k}) \ge \delta >
0$, uniformly in $k$. For each $k$, we will choose the curve $\beta_k$
which will satisfy \eqref{eqn-claim-111}, from a constructed family of
curves $\{\beta_{\alpha}^k\}$ that foliate our solution $(M,g(t_k))$.
Define the foliation of $(M,g(t_k))$ by the curves
$\{\beta_{\alpha}^k\}$ so that for every $\alpha$ and every $x\in
\beta_{\alpha}^k$, $\dist_{t_k}(x,y_k) = \alpha$. Choose a curve
$\beta_k$ from that foliation so that the corresponding curve
$\tilde{\beta}_k$ on $S^2$ splits $S^2$ in two parts of equal areas, where the area
is computed with respect to the round metric.

Since the isperimetric constant for the sphere $I_{S^2} = 1$, that is 
$$1 \le L_{S^2}(\tilde{\beta}_k)\, (\frac{1}{A_1} + \frac{1}{A_2}) = L_{S^2}(\tilde{\beta}_k)\,  \frac{4}{A_{S^2}}$$
we have 
$$L_{S^2}(\tilde{\beta}_k) \ge \delta > 0\qquad \mbox{for all} \,\, k.$$

To finish the proof of  the theorem  we will now show that there  exists a uniform constant $C$  so that 
$$L_k(\beta_k) \le C\qquad \mbox{for all} \,\, k.$$
To this end, we observe first that the area element of $g(t_k)$, when  computed in polar coordinates,  is
$$da_k = J_k(r,\theta)\, r\, dr\,  d\theta$$
where $J_k(r,\theta)$ is the Jacobian and $r$ is the radial distance from $y_k$.
The length of $\beta_r^k$ is given by
$$L_k^r = \int_0^{2\pi}J_k(r,\theta)\, r\, d\theta$$
which implies that 
$$\frac{L_k^r}{r}=\int_0^{2\pi}J_k(r,\theta)\, d\theta.$$
By the Jacobian comparison theorem, for each fixed $\theta$,  we have
\begin{equation}
\label{eq-jac-comp}
\frac{J_k'(r,\theta)}{J_k(r,\theta)} \le \frac{J_a'(r,\theta)}{J_a(r,\theta)}
\end{equation}
where the derivative is in the $r$ direction and $J_a(r,\theta)$
denotes the Jacobian for the model space and $a$ refers to a lower
bound on Ricci curvature (the model space is a simply connected space of constant sectional curvature equal to $a$). In our case $a = 0$ (since $R \geq 0$) and
the model space is the euclidean plane, which implies that the right
hand side of (\ref{eq-jac-comp}) is zero and therefore $J_k(r,\theta)$
decreases in $r$. Hence ${L_k^r}/{r}$ decreases in $r$. 

In the proof
of Lemma \ref{lem-ar-as} we showed that there are uniform constants
$C_1, C_2$ so that
$$C_1|t_k| \le \rho_k \le C_2|t_k| \qquad \mbox{and} \qquad c_1|t_k| \le \sigma_k \le C_2|t_k|.
$$
We have shown that $\gamma_{t_k} \to \gamma_{\infty}$ and
$\gamma_{\infty}$ is a circle in $M_{\infty}$.  Let $y_k, p_k$ be the
points which we have chosen previously.  We may assume that
$\dist_{t_k}(y_k,p_k) = \sigma_k$. Choose a curve $\bar{\gamma}_k \in
M$ so that $p_k \in \bar{\gamma}_k$ and that for every $x\in
\bar{\gamma}_k$ we have the $\dist_{t_k}(y_k,x) = \sigma_k$. Observe
that for every $x\in \gamma_{t_k}$, by the figure we have
$$\sigma_k \le \dist_{t_k}(x,y_k) \le \sigma_k + \frac{C}{\sigma_k},$$
for sufficiently big $k$. For $x\in \gamma_{t_k}$ let $z = \bar{\gamma}_{t_k}\cap\overline{y_kx}$.  
Then
$$\dist_{t_k}(x,\bar{\gamma}_k) \le \dist_{t_k}(x,z) \le \dist_{t_k}(y_k,x) -
\dist(y_k,z) \le \sigma_k + \frac{C}{\sigma_k} - \sigma_k =
\frac{C}{\sigma_k}.$$ 
This implies that the curves $\bar{\gamma}_k$
converge to $\gamma_{\infty}$ as $k\to\infty$.  Moreover, this also
implies the curve $\bar{\gamma}_k$ is at distance $\sigma_k =
O(|t_k|)$ from $y_k$ and if $s_k = \dist_{t_k}(\beta_k,y_k)$, then
$s_k = O(|t_k|)$ and we also know $L_k(\bar{\gamma}_k) \le C$, for all
$k$.  We may assume $s_k \le \sigma_k$ for infinitely many $k$, otherwise
we can consider point $x_k$ instead of $y_k$ and do the same analysis as above
but with respect to $x_k$. Since $J_k(r,\theta)$ decreases in $r$ we have
$$\frac{L_k^{s_k}}{s_k} \le \frac{L_k^{\sigma_k}}{\sigma_k}$$
that is
$$L_k(\beta_k) = L_k^{s_k}  \le \frac{s_k}{\sigma_k}\, L_k^{\sigma_k} = \frac{s_k}{\sigma_k}\,L_k(\bar{\gamma}_k)
\le C\qquad \mbox{for all} \,\,\, k$$
finishing the proof of \eqref{eqn-claim-111} and the theorem. 

\end{proof}

Based on the arguments of the proof of Theorem   \ref{prop-sphere}, we will show the following lemma, which was used in the proof  of Theorem \ref{prop-limit}.

\begin{lem}\label{prop-plane}  Assuming that our evolving metric $g(t)=\bar u \, g_e$, where $g_e$ denotes the standard euclidean metric, it  is impossible to have that the backward limit 
$$\bar u_\infty:= \lim_{t \to -\infty} \bar u(\cdot,t) = \gamma$$
for a constant $\gamma >0$. 
\end{lem}
\begin{proof}
We will use the arguments from the proof of Theorem  \ref{prop-sphere} presented above. For a given time $t$,
which will be chosen sufficiently close to $-\infty$,  we denote  by $(M,g)$ our evolving  surface at time $t$ (for simplicity we omit to write $t$ in all considered quantities below in the proof of the Lemma) and by
$\gamma$ the isoperimetric   curve which divides $M$ into two regions $M_1$ and $M_2$. 
We have seen in the proof of Theorem  \ref{prop-sphere}   that $$L_{g}(\gamma) \leq C$$
for a uniform constant $C$ and that the areas of $M_1$ and $M_2$ are comparable to $|t|$. 

Let ${\mathcal R}  >0$ be a large but uniform in time constant, which   will be chosen in the sequel. 
By our assumption, there exist a $t_0 < 0$ and  a point $O \in M$, 
so  that the metric $g$ is very close to the flat metric on the ball $B_{\mathcal R}(O)$  which is taken with respect to the metric 
$g$, for $t \leq t_0 <0$. Notice that since $g$ is  very close to the flat metric, $B_{\mathcal R}(O)$ is also close to the  Euclidean ball. 

We may assume, without loss of generality, that $O \in M_2$. Since $M_1$ and $M_2$ have unbounded areas,
as $|t| \to \infty$, the curve $\gamma$ cannot be entirely contained in $B_{\mathcal R}(O)$. Hence, 
$\gamma \cap B_{\mathcal R}(O)^c \neq \emptyset$. By choosing ${\mathcal R}$ larger than $2C$, we then have that 
$$\gamma \cap B_{ {\mathcal R}/2} (O) = \emptyset.$$
As in the proof of Theorem  \ref{prop-sphere}, consider the point $x \in M_1$ which is the furthest from $\gamma$
and the family  of curves $\beta_r$ of radial distance $r$ from $x$ which foliate our surface $M$.  Let $\sigma = 
\dist_g (x,\gamma) = \dist_g(x,p)$, for some point $p\in \gamma$.  Let $\bar{\gamma}$ be the curve such that $p\in \bar{\gamma}$ and such that for all  $y\in \bar{\gamma}$ we have  $\dist_g(p,y) = \sigma$. As in the proof of 
Theorem \ref{prop-sphere} we have $L_{g}(\bar{\gamma}) \le C$, 
$$\dist_g(y,\bar{\gamma}) \le \frac{C}{\sigma}$$ for all $y\in \gamma$ and $\sigma$ is comparable to $|t|$. These all together,  combined with the fact that   $\gamma\cap B_{\mathcal{R}/2}(O) = \emptyset$,  imply that $\bar{\gamma}\cap B_{\mathcal{R}/2}(O) = \emptyset$, for $ |t| \ge |t_0|$ and $|t_0|$ chosen sufficiently large. Let $\beta_{r_1}$ be the curve that contains the point $O$.
Based on the previous analysis, since all the curves in the foliation $\{\beta_r\}_{r\ge 0}$ of our surface $M$ are mutually disjoint, we conclude $r_1 > \sigma$. In the proof of Theorem \ref{prop-sphere} we have argued  that ${L_{g}^r}/{r}$ decreases in $r$. This implies
$$\frac{L_g^r}{r} \le \frac{L_g^{\sigma}}{\sigma} = \frac{L(\bar{\gamma})}{\sigma}$$
finally yielding  the bound 
$$L_g(\beta_{r_1}) \le C\, \frac{r_1}{\sigma} \le \tilde{C}$$
for a uniform constant $\tilde{C}$, since $\sigma$ is comparable to $|t|$ and $r_1 \le \diam(M,g) \le C|t|$.

If $\beta_{r_1} \cap \partial  B_{ {\mathcal R}/4} (O) \neq \emptyset$, then $L_g(\beta_{r_1}) \geq {\mathcal R}/4$ which will lead to a contradiction if we choose ${\mathcal R} > 4 \bar C$. 
Otherwise, $\beta_{r_1} $ is entirely  contained in $B_{ {\mathcal R}/4} (O)$ which means that there exists another curve 
$\beta_{r_2}$,  which encloses  $\beta_{r_1}$ and  is contained in the  closure of  $B_{ {\mathcal R}/4} (O)$ 
and touches the boundary of $B_{ {\mathcal R}/4} (O)$. Since our metric
on $B_{\mathcal R}(O)$ is very close to the Euclidean metric this would imply that $L_g(\beta_{r_2}) > {\mathcal R}/8$ which would
also lead to a contradiction if we choose ${\mathcal R} > 8\bar C$.  This finishes the proof of the lemma.

\end{proof}



\end{document}

%% file: slika.pstex_t
\begin{picture}(0,0)%
\includegraphics{slika.pstex}%
\end{picture}%
\setlength{\unitlength}{1973sp}%
\begingroup\makeatletter\ifx\SetFigFont\undefined%
\gdef\SetFigFont#1#2#3#4#5{%
  \reset@font\fontsize{#1}{#2pt}%
  \fontfamily{#3}\fontseries{#4}\fontshape{#5}%
  \selectfont}%
\fi\endgroup%
\begin{picture}(8577,6570)(286,-6673)
\put(1276,-4636){\makebox(0,0)[lb]{\smash{{\SetFigFont{7}{8.4}{\familydefault}{\mddefault}{\updefault}$\gamma_{\infty}$}}}}
\put(2026,-4036){\makebox(0,0)[lb]{\smash{{\SetFigFont{7}{8.4}{\familydefault}{\mddefault}{\updefault}$1$}}}}
\put(1426,-3661){\makebox(0,0)[lb]{\smash{{\SetFigFont{7}{8.4}{\familydefault}{\mddefault}{\updefault}$\pi$}}}}
\put(301,-2386){\makebox(0,0)[lb]{\smash{{\SetFigFont{7}{8.4}{\familydefault}{\mddefault}{\updefault}w}}}}
\put(1651,-2011){\makebox(0,0)[lb]{\smash{{\SetFigFont{7}{8.4}{\familydefault}{\mddefault}{\updefault}$\sqrt{w^2+\pi^2}$}}}}
\put(1426,-286){\makebox(0,0)[lb]{\smash{{\SetFigFont{7}{8.4}{\familydefault}{\mddefault}{\updefault}$\beta_w$}}}}
\put(7083,-4176){\makebox(0,0)[lb]{\smash{{\SetFigFont{7}{8.4}{\familydefault}{\mddefault}{\updefault}$p_i$}}}}
\put(7064,-2945){\makebox(0,0)[lb]{\smash{{\SetFigFont{7}{8.4}{\familydefault}{\mddefault}{\updefault}$\beta_w^i$}}}}
\put(7028,-4884){\makebox(0,0)[lb]{\smash{{\SetFigFont{7}{8.4}{\familydefault}{\mddefault}{\updefault}$\gamma_i$}}}}
\put(4576,-3811){\makebox(0,0)[lb]{\smash{{\SetFigFont{7}{8.4}{\familydefault}{\mddefault}{\updefault}$\phi_i$}}}}

\end{picture}%